\documentclass[12pt]{amsart}
\usepackage{amsfonts}
\usepackage{mathrsfs}
\usepackage{amsmath}
\usepackage{graphicx}
\usepackage{amssymb}
\usepackage{amscd}
\usepackage{color}
\usepackage{xcolor}
\usepackage[percent]{overpic}
\usepackage[all]{xy}
\usepackage{imakeidx}
\usepackage{hyperref}
\usepackage[foot]{amsaddr}

\makeindex

\newtheorem{thm}{Theorem}[section]
\newtheorem{prop}[thm]{Proposition}
\newtheorem{lemma}[thm]{Lemma}

\newtheorem{cor}[thm]{Corollary}
\newtheorem{conj}[thm]{Conjecture}

\theoremstyle{definition}

\newtheorem{definition}[thm]{Definition}

\newtheorem{remark}[thm]{Remark}

\numberwithin{equation}{section}

\begin{document}

\title[Symplectomorphism group of $T^*(G_\mathbb{C}/B)$ and the braid group I]{Symplectomorphism group of $T^*(G_\mathbb{C}/B)$ and the braid group I: a homotopy equivalence for $G_\mathbb{C}=SL_3(\mathbb{C})$}
\author[Xin Jin]{Xin Jin}
\email{xinjin2013@gmail.com}

\address{Department of Mathematics, Northwestern University, Evanston, IL}

\subjclass[2000]{}

\date{}


\begin{abstract}
For a semisimple Lie group $G_\mathbb{C}$ over $\mathbb{C}$, we study the homotopy type of the symplectomorphism group of the cotangent bundle of the flag variety and its relation to the braid group. We prove a homotopy equivalence between the two groups in the case of $G_\mathbb{C}=SL_3(\mathbb{C})$, under the $SU(3)$-equivariance condition on symplectomorphisms.
\end{abstract}

\maketitle

\tableofcontents

\section{Introduction}
For a semisimple Lie group $G_\mathbb{C}$ over $\mathbb{C}$, the cotangent bundle of the flag variety $T^*\mathcal{B}$ and its relation to the braid group have led to numerous active research  directions in geometric representation theory, algebraic geometry and symplectic topology.  The main driving force for these is due to the fruitful structures underlying the Springer resolutions and the adjoint quotient maps. 

This paper is an attempt to study the homotopy type of the symplectomorphism group of $T^*\mathcal{B}$ and its relation to the braid group, from a purely geometric point of view. We especially focus on the case of $G_\mathbb{C}=SL_3(\mathbb{C})$.

\subsection{Motivation and set-up}
The motivation is from the (strong) categorical braid group action on $D(\mathcal{B})$, the derived category of constructible sheaves on $\mathcal{B}$, by Deligne \cite{Deligne} and Rouquier\cite{Rouquier}. This action gives rise to $G_\mathbb{C}$-equivariant automorphisms of $D(\mathcal{B})$. One can translate the result to symplectic geometry via the Nadler-Zaslow correspondence\cite{NZ}. Recall that the Nadler-Zaslow correspondence 
gives a categorical equivalence between $D(X)$ and $DFuk(T^*X)$, the derived Fukaya category of $T^*X$, for any compact analytic manifold $X$ (see Section \ref{motiv} for more details). Since symplectomorphisms of $T^*\mathcal{B}$ with reasonable behavior near infinity induce automorphisms of $DFuk(T^*\mathcal{B})$, it is natural to form the following conjecture.
\begin{conj}\label{conj1}
The ``$G_\mathbb{C}$-equivariant" symplectomorphism group of $T^*\mathcal{B}$ is homotopy equivalent to the braid group. 
\end{conj}

To rigorously state the conjecture, one has to give a definition of ``$G_\mathbb{C}$-equivariance" on symplectomorphisms. The global $G_{\mathbb{C}}$-equivariance condition on a symplectomorphism would force it to be the identity. The reason is the following.  
The Springer resolution (see (\ref{Springer}) for the definition)
$$\mu_\mathbb{C}: T^*\mathcal{B}\rightarrow \mathcal{N},$$
gives a $G_\mathbb{C}$-equivariant isomorphism from the dense $G_\mathbb{C}$-orbit in $T^*\mathcal{B}$ to $\mathcal{N}_{reg}$, the orbit of regular nilpotent elements in $\mathcal{N}$. Suppose $\varphi$ is a $G_\mathbb{C}$-equivariant symplectomorphism, then the graph of $\varphi|_{\mu_\mathbb{C}^{-1}(\mathcal{N}_{reg})}$ is a complex Lagrangian submanifold in $T^*\mathcal{B}^-\times T^*\mathcal{B}$, hence the graph of $\varphi$ is a closed complex Lagrangian. Therefore, $\varphi$ preserves the holomorphic symplectic form and then preserves $\mu_\mathbb{C}$ (see Lemma \ref{inv_moment}), so we can conclude that $\varphi=\mathrm{id}$. 

A natural replacement of the global $G_\mathbb{C}$-equivariance condition is to require $\varphi$ to be \emph{$G_\mathbb{C}$-equivariant at infinity}. It can be formulated via the Lagrangian correspondence $L_\varphi$, i.e. the graph of $\varphi$, in $T^*\mathcal{B}^-\times T^*\mathcal{B}\simeq T^*(\mathcal{B}\times\mathcal{B})$ and its relation to the Steinberg variety $\mathcal{Z}$. Recall that the Steinberg variety $\mathcal{Z}$ is the union of the conormal varieties to the $G_\mathbb{C}$-orbits in $\mathcal{B}\times \mathcal{B}$ under the diagonal action. Using the $\mathbb{R}_+$-action on $T^*(\mathcal{B}\times\mathcal{B})$, one can projectivize the cotangent bundle and get a compact symplectic manifold with a contact boundary. We denote the boundary by 
 $T^\infty(\mathcal{B}\times\mathcal{B})$, and for any Lagrangian $L$ in the cotangent bundle, we use $L^\infty$ to denote for $\overline{L}\cap T^\infty(\mathcal{B}\times\mathcal{B})$. Then we make the following definition (see Section \ref{G_Sympl} for more discussions).
\begin{definition}\label{Def1}
A symplectomorphism $\varphi$ of $T^*\mathcal{B}$ is \emph{$G_\mathbb{C}$-equivariant at infinity} if $L_\varphi^\infty\subset \mathcal{Z}^\infty$. We denote by 
$\mathrm{Sympl}_\mathcal{Z}(T^*\mathcal{B})$ for the group of symplectomorphisms that are  $G_\mathbb{C}$-equivariant at infinity.
\end{definition}
We are content with this definition since the Steinberg variety is one of the key players in geometric representation theory, and  this definition builds a natural bridge between geometric representation theory and symplectic geometry. 

For example, if $G_\mathbb{C}=SL_2(\mathbb{C})$, then the symplectomorphisms that we are considering are the compactly supported ones. For general $G_\mathbb{C}$, $\varphi$ has to preserve the Springer fibers, i.e. fibers of $\mu_\mathbb{C}$, at infinity. If we fix a maximal compact subgroup $G$ in $G_\mathbb{C}$ (e.g. $SU(n)$ inside $SL_n(\mathbb{C})$ and identify $\mathcal{B}$ with $G/T$), then we can consider the subgroup $\mathrm{Sympl}_\mathcal{Z}^G(T^*\mathcal{B})$ of (genuine) $G$-equivariant symplectomorphisms. We make the following conjecture.
\begin{conj}\label{conjG}
There is a sequence of homotopy equivalences 
$$\mathrm{Sympl}_\mathcal{Z}(T^*\mathcal{B})\simeq \mathrm{Sympl}_\mathcal{Z}^G(T^*\mathcal{B})\simeq B_\mathbf{W}.$$
\end{conj}
Our main results provide evidence for this conjecture.

\subsection{Main Theorem}

\begin{thm}\label{intr_thm1}
(1) There is a natural surjective group homomorphism $$\beta_G: \mathrm{Sympl}_\mathcal{Z}^G(T^*\mathcal{B})\rightarrow B_{\mathbf{W}}, \text{ for }G=SU(n).$$\\
(2) $\beta_G$ is a homotopy equivalence for $G=SU(2), SU(3)$.
\end{thm}
The construction of $\beta_G$ is purely geometric as apposed to the alternative categorical construction (see Remark \ref{cat_braid} below). As mentioned before, every $\varphi\in\mathrm{Sympl}^G_\mathcal{Z}(T^*\mathcal{B})$ must preserve each reduced space of the Hamiltonian $G$-action. So the problem of studying the (weak) homotopy type of $\mathrm{Sympl}^G_\mathcal{Z}(T^*\mathcal{B})$ can be roughly reduced to the study of homotopy classes (and homotopy between homotopies and so on) of the symplectomorphisms on the Hamiltonian reductions over a Weyl chamber $W$ in the dual of the Cartan subalgebra $\mathfrak{t}^*\cong i\mathfrak{t}$, with some further restrictions at infinity.  

For $n=2$, the reduced space over each element $p\in W$ is a point. However, we have to divide them into two cases. If $p\neq 0$, then $\mu^{-1}(p)$ is an orbit of the $T$-action, so $\varphi|_{\mu^{-1}}(p)$ is a rotation and corresponds to an element in $S^1$. If $p=0$, then the restriction of $\varphi$ on $\mu^{-1}(0)=T^*_\mathcal{B}\mathcal{B}$ is a $G$-equivariant automorphism of $G/T$. Since $\mathrm{Aut}^G(G/T)\cong \mathbf{W}$, $\varphi|_{\mu^{-1}(0)}$ corresponds to an element in $\mathbf{W}\cong\mathbb{Z}_2$. Note that the circles over the interior of $W$ approach the zero section to a big circle, we see that $\varphi$ corresponds to a path in $S^1$ starting from $\pm 1$ and ending at $1$, and that $\varphi$ is a (iterated) Dehn twist (see Figure \ref{Dehn_twist_SL_2}). It is then easy to see that $\mathrm{Sympl}_\mathcal{Z}^G(T^*\mathcal{B})$ is homotopy equivalent to $B_2=\mathbb{Z}$, which is homotopy equivalent to $\mathrm{Sympl}^c(T^*S^2)$, by the result of Seidel\cite{Seidel}.

For $n\geq 3$, things are more interesting and we will not have all $\varphi$ being compactly supported. The picture in the case of $G=SU(3)$ is very illustrating. Let $\mu: T^*\mathcal{B}\rightarrow i\mathfrak{su}(3)\cong \mathfrak{su}(3)^*$ be the moment map. Along the ray generated by $p=\mathrm{diag}(1,0,-1)\in i\mathfrak{t}$, the reduced spaces are all $S^2$ with three distinguished points corresponding to the singular loci of $\mu$. There are exactly two types Springer fibers contained in $\mu^{-1}(p)$: one is the Springer fiber over a regular nilpotent element (a $3\times 3$-nilpotent matrix having one single Jordan block in its Jordan normal form), which is just a point; the other is the Springer fiber over a subregular nilpotent element (a $3\times 3$-nilpotent matrix having two Jordan blocks), which is the wedge of two $2$-spheres. The union of subregular Springer fibers in $\mu^{-1}(p)$ projects to two line segments connecting the three special points in the reduced space $M_p$ (see Figure \ref{Dehn_twist_SL_3}). Now we draw a small disc $\mathcal{U}_s$ around these line segments in $\mu^{-1}(s\cdot p)$ for each $s>0$, which forms a  $\mathbb{R}_+$-invariant family. Let $\varphi_s$ be the induced map on $M_{s\cdot p}$ by $\varphi$. As $s\rightarrow\infty$, $\varphi_s$ tends to fix all the points outside of $\mathcal{U}_s$, hence after a small homotopy near $\partial (\mathcal{U}_s)$, $\varphi_s|_{\mathcal{U}_s}$ becomes a symplectomorphism of $\mathcal{U}_s$, which permutes the three marked points and fixes each point on the boundary. Therefore, it gives rise to an element in $B_3$, the braid group of three strands.
\begin{figure}
\centering
  \begin{overpic}[width=3.5in]{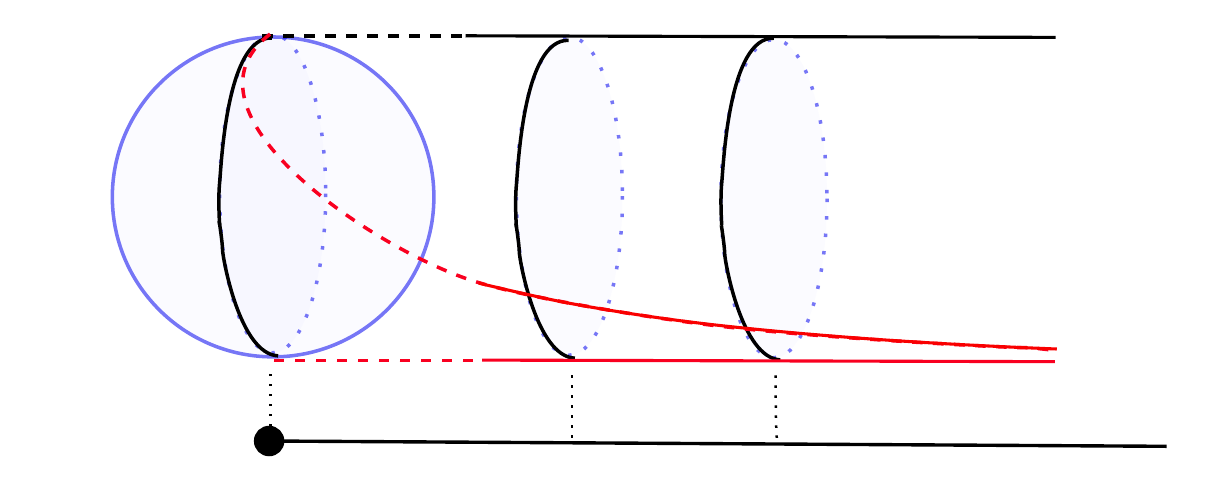}
   \put(100,0){$W$}
\end{overpic}
\caption{The fibers of $\mu$ over a Weyl chamber $W$ for $G=SU(2)$ and the Dehn twist. The Dehn twist moves the lower straight red line to the upper red curve. }\label{Dehn_twist_SL_2}
\end{figure}

\begin{figure}
\centering
  \begin{overpic}[width=3.5in]{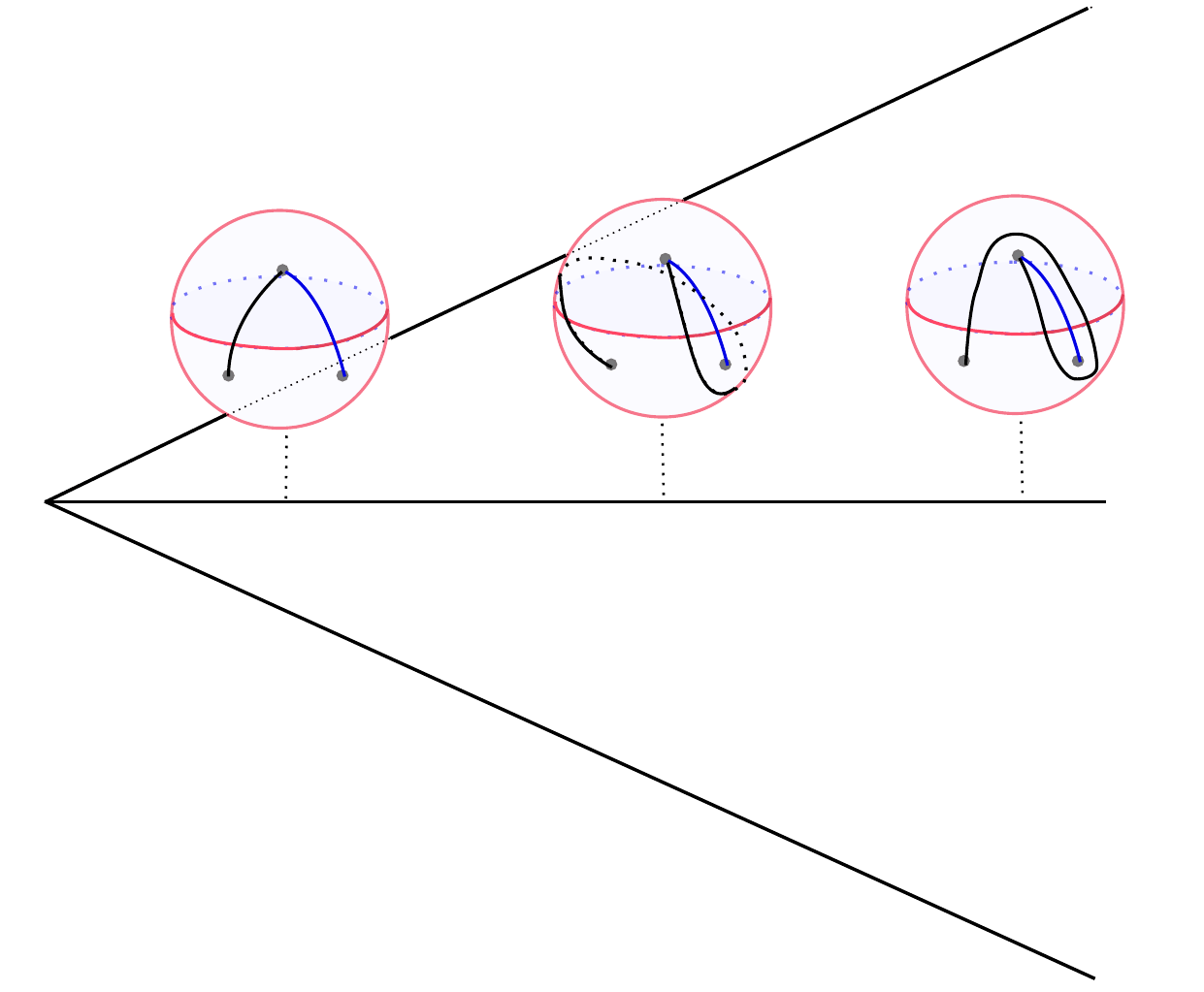}
   \put(100,40){$\mathbb{R}_+\cdot p$}
    \put(68,27){$W$}
\end{overpic}
\caption{The reduced spaces over $\mathbb{R}_+\cdot p, p=\mathrm{diag}(1,0,-1)\in i\mathfrak{t}$, and an illustration of one symplectomorphism for $G=SU(3)$. The reduced spaces have been rescaled to be of the same size. The union of the two arcs in the leftmost reduced space is the projection of the subregular Springer fibers. The symplectomophism restricts to the identity near the zero section of $T^*\mathcal{B}$, so fixes every point in the reduced spaces near the vertex of $W$. The arcs in the two reduced spaces on the right illustrate how the symplectomorphism moves the Springer fibers.}\label{Dehn_twist_SL_3}
\end{figure}

For $G=SU(n)$, we focus on certain region in $\mu^{-1}(p_n)$, where $p_n=\mathrm{diag}(1,-1,0,...,0)\in i\mathfrak{su}(n)$, and use similar argument. To prove surjectivity of $\beta_G$, we explicitly construct fiberwise Dehn twists associated to each simple root $\alpha$ (see Remark \ref{relate} below), and we show that their image under $\beta_G$ generates  $B_\mathbf{W}$. 
 
\begin{remark}\label{cat_braid}
One could compare the map in Theorem \ref{intr_thm1} with the composition $\mathrm{Sympl}_\mathcal{Z}^G(T^*\mathcal{B})\rightarrow \mathrm{Aut}(DFuk(T^*\mathcal{B}))\cong \mathrm{Aut}(D(\mathcal{B}))$ through the categorical action of $\mathrm{Sympl}_\mathcal{Z}^G(T^*\mathcal{B})$. Conjecture \ref{conjG} implies that this construction gives $\beta_G$ as well. The reason is that the Lagrangian correspondences for the fiberwise Dehn twists in $T^*(\mathcal{B}\times\mathcal{B})$ represent exactly the integral kernels for the braid group action on $D(\mathcal{B})$. 
\end{remark}

For part (2) of Theorem \ref{intr_thm1}, we have seen the proof when $G=SU(2)$. The proof for $G=SU(3)$ consists of two steps. The first step is to construct local symplectic charts for $\mu^{-1}(W)$ and ``trivialize" each chart by certain reduced spaces. The main techniques are the Duistermaat-Heckman theorem on the normal form of a moment map near a regular value (see \cite{GuSt89}), and Weinstein's Lagrangian tubular neighborhood theorem. The second step is to find the homotopy type of the symplectomorphism groups over the local charts by constructing various fibrations, and then realize $\ker\beta_G$ as the fiber product of these spaces. 
One of the difficulties along the way is to take special care for the singular loci of the moment map. 

\begin{remark}\label{relate}
 This is a remark on some related result by Seidel-Smith and Thomas.
Seidel-Smith \cite{SeSm} considered symplectic fibrations that naturally arise in the adjoint quotient maps in Lie theory, and constructed link invariants by the symplectic monodromies.  It is described in \cite{Tho} that the braid group actions are exactly the ``family Dehn twists" about the family of isotropic spheres over $T^*(G_\mathbb{C}/P)$, which are the image of the left map in the standard correspondence
$$T^*(G_\mathbb{C}/B)\leftarrow G_\mathbb{C}/B\times_{G_\mathbb{C}/P} T^*(G_\mathbb{C}/P)\rightarrow T^*(G_\mathbb{C}/P),$$
associated to the $\mathbb{P}^1$-fibration $G_\mathbb{C}/B\rightarrow G_\mathbb{C}/P$, for a minimal parabolic subgroup $P$. This is essentially the same as the fiberwise Dehn twists that we consider here, though we identify $T^*\mathcal{B}$ as a symplectic fiber bundle over $T^*(G_\mathbb{C}/P)$ using the Killing form on $\mathfrak{g}$ (rather than $\mathfrak{g}_\mathbb{C}$), and we explicitly make the fiberwise Dehn twists all $G$-equivariant.  
\end{remark}

\subsection{Acknowledgement} 
I am very grateful to my PhD advisor Prof. David Nadler for guiding me to this topic, and for invaluable discussions and consistent encouragement. I have benefited a lot from the discussions with Prof. Allen Knutson. I would also like to thank Prof. Denis Auroux, Ivan Losev, Vivek Shende, David Treumann, Zhiwei Yun, Eric Zaslow, and Dr. Long Jin, Penghui Li for their interest in this work and helpful conversations.  Special thanks go to the anonymous referee for many useful comments and suggestions, which improved the paper significantly. This work formed part of my PhD thesis at the University of California at Berkeley. 

\section{Preliminaries and Set-ups}
\emph{Notations:} Throughout this paper, we will use $G_\mathbb{C}$ to denote a semisimple Lie group over $\mathbb{C}$, with Lie algebra $\mathfrak{g}_\mathbb{C}$, and $G$ to denote for a maximal compact subgroup in $G_\mathbb{C}$ with Lie algebra $\mathfrak{g}$. We will mostly focus on type $A$, e.g. $G_\mathbb{C}=SL_n(\mathbb{C})$ and $G=SU(n)$. Fix a Borel subgroup $B$ in $G_{\mathbb{C}}$ with Lie algebra $\mathfrak{b}$ and nilradical $\mathfrak{n}$, and let $\mathcal{B}$ denote for $G_\mathbb{C}/B$. Then $T:=B\cap G$ is a maximal torus in $G$, with Lie algebra $\mathfrak{t}$, and we have the canonical identification $\mathcal{B}\cong G/T$. For $G_\mathbb{C}=SL_n(\mathbb{C})$, we will mostly take $B$ to be the subgroup of upper triangular matrices, then $T$ consists of diagonal matrices in $SU(n)$.

\subsection{Set-up for the symplectomorphism group}
We consider $T^*\mathcal{B}$ as a \emph{real} symplectic manifold, and would like to study the homotopy type of its symplectomorphism group. Since $T^*\mathcal{B}$ is noncompact, we must put some restrictions on the behavior of the symplectomorphisms near the infinity of $T^*\mathcal{B}$, so that the resulting group has ``nice" structures. A typical restriction is to make the  symplectomorphisms compactly supported, which will turn out to be too small (see the discussion below). Instead we pose the condition that the symplectomorphisms are $G_\mathbb{C}$-equivariant at infinity, where the $G_\mathbb{C}$-action is the standard Hamiltonian action induced from the left action of $G_\mathbb{C}$ on $\mathcal{B}$. We will make the restriction more precise after a brief discussion of the motivation.

\subsubsection{Motivation for the definition}\label{motiv} Let $D(\mathcal{B})$ be the constructible derived category of sheaves on $\mathcal{B}$, and let $DFuk(T^*\mathcal{B})$ be the derived Fukaya category of $T^*\mathcal{B}$. There is a categorical equivalence (the Nadler-Zaslow correspondence) between $D(M)$ and $DFuk(T^*M)$, for any real analytic manifold $M$. Motivated by the results of \cite{Deligne} and \cite{Rouquier} on the braid group action on $D(\mathcal{B})$, which are $G_\mathbb{C}$-equivariant automorphisms of the category, and the Nadler-Zaslow correspondence between $D(\mathcal{B})$ and $DFuk(T^*\mathcal{B})$, we would like to study the group of  ``$G_\mathbb{C}$-equivariant" symplectomorphisms of $T^*\mathcal{B}$, and to see its relation to the braid group. As discussed in the Introduction, the most natural interpretation of  ``$G_\mathbb{C}$-equivariancy" is to impose that $\varphi$ is $G_\mathbb{C}$-equivariant at infinity. 

\subsubsection{Definition of $\mathrm{Sympl}_{\mathcal{Z}}^{G}(T^*\mathcal{B})$}\label{G_Sympl} Let $\varphi$ be any symplectomorphism of $T^*\mathcal{B}$, then its graph $L_\varphi$ is a Lagrangian correspondence in $(T^*\mathcal{B})^-\times T^*\mathcal{B}\cong T^*(\mathcal{B}\times\mathcal{B})$. Using the $\mathbb{R}_+$-action on the cotangent fibers of $T^*(\mathcal{B}\times\mathcal{B})$, we can projectivize the space with the boundary divisor $T^\infty(\mathcal{B}\times\mathcal{B})$ being a contact manifold, with contact form $\theta^\infty$. We require $\varphi$ to be well-behaved near the infinity divisor, in the sense that $L_\varphi^\infty:=\overline{L_\varphi}\cap T^\infty(\mathcal{B}\times\mathcal{B})$ is $\theta^\infty$-isotropic.

As discussed in the Introduction, global $G_\mathbb{C}$-equivariancy on a symplectomorphism $\varphi$ forces $\varphi$ to preserve each Springer fiber, which implies that $\varphi$ must be the identity. However, if we only require the $G_\mathbb{C}$-equivariancy condition ``at infinity", this would give a reasonable constraint on $\varphi$ by 
$$L_\varphi^\infty\subset (\widetilde{\mathcal{N}}\times_\mathcal{N} \widetilde{\mathcal{N}})^\infty,$$
where $\widetilde{\mathcal{N}}=T^*\mathcal{B}$ and the fiber product is taken for the Springer resolution $\mu_\mathbb{C}$. Note that $\widetilde{\mathcal{N}}\times_\mathcal{N} \widetilde{\mathcal{N}}$  is just the \emph{Steinberg variety} $\mathcal{Z}$, which is a Lagrangian subvariety by an alternative description as the union of conormal varieties to the diagonal  $G_\mathbb{C}$-orbits $\mathcal{O}_w, w\in\mathbf{W}$ in $\mathcal{B}\times\mathcal{B}$ (here $\mathcal{O}_{\mathbf{1}}=\Delta_{\mathcal{B}}$). Thus, we make the following definition, which adds  to Definition \ref{Def1} a partially compactly supported requirement for $\varphi$. 

\begin{definition}\label{def G_C}
A symplectomorphism $\varphi$ of $T^*\mathcal{B}$ is \emph{$G_\mathbb{C}$-equivariant at infinity} if \\
(1) $L_\varphi^\infty\subset \mathcal{Z}^\infty$, \\
(2) (Partially compactly supported) There is an open neighborhood of $\mathcal{Z}^\infty-\bigcup\limits_{w\in\mathbf{W}-\{\mathbf{1}\}}
\overline{T^\infty_{\mathcal{O}_w}(\mathcal{B}\times\mathcal{B})}$ in $\overline{L}_{\varphi}$ that is contained in $\overline{T}^*_{\mathcal{O}_{\mathbf{1}}}(\mathcal{B}\times\mathcal{B})$;\\

We denote by 
$\mathrm{Sympl}_\mathcal{Z}(T^*\mathcal{B})$ for the group of symplectomorphisms with $G_\mathbb{C}$-equivariancy at infinity. 
\end{definition}

We define the $C^\infty$-topology on $\mathrm{Sympl}_\mathcal{Z}(T^*\mathcal{B})$ as follows.
\begin{align}\label{C^k-top}
&\lim\limits_{n\rightarrow\infty} f_n=f\in \mathrm{Sympl}_\mathcal{Z}(T^*\mathcal{B})\Leftrightarrow\\ \nonumber &\text{(a)} \lim\limits_{n\rightarrow\infty} f_n|_K=f|_K\text{ in }C^\infty(K, T^*\mathcal{B})\text{ for all compact subdomain } K;\\ \nonumber
&\text{(b) for any sequence of points }y_n\in L_{f_n}, \text{ if }\lim\limits_{n\rightarrow\infty}y_n \text{ exists in }T^\infty(\mathcal{B}\times\mathcal{B}), \text{ then it lies in }\mathcal{Z}^\infty. \nonumber
\end{align}

It is easy to see that $\mathrm{Sympl}_\mathcal{Z}(T^*\mathcal{B})$ endowed with this topology is a topological group. The main concern about the topology of symplectomorphisms on a non-compact symplectic manifold $M$ is that the induced automorphisms on the Fukaya category $Fuk(M)$ of a continuous family of symplectomorphisms should remain the same, i.e. there should be a well defined map $\pi_0(\mathrm{Sympl}(M))\rightarrow \mathrm{Aut}(Fuk(M))$. In our setting, we view each $\varphi\in \mathrm{Sympl}_\mathcal{Z}(T^*\mathcal{B})$ as a Lagrangian correspondence $L_\varphi$ with $L_\varphi^\infty\subset\mathcal{Z}^\infty$, and it corresponds to a sheaf (or an integral kernel) $\mathcal{F}_\varphi$ in $Sh_\mathcal{Z}(\mathcal{B}\times \mathcal{B})$, the full subcategory of sheaves with singular support contained in $\mathcal{Z}$ (cf. \cite{NZ}, \cite{Nadler2}).  Now if we have a continuous family $\{\varphi_s\}_{0\leq s\leq 1}$ in the $C^1$-topology defined by (\ref{C^k-top}), then the family of sheaves $\mathcal{F}_{\varphi_s}$ remain the same. This can be argued using the test branes representing the micolocal stalk functors in $Fuk(T^*\mathcal{B})$ and the fact that the isotopy of the branes $L_{\varphi_s}$ is non-characteristic with respect to any fixed finite set of test branes; for more details see \cite{Jin15} and \cite{Nadler2}.

We also consider the subgroup of $G$-equivariant symplectomorphisms,  denoted as $\mathrm{Sympl}^G_\mathcal{Z}(T^*\mathcal{B})$. As stated in the Introduction, we conjecture  that 
$$\mathrm{Sympl}_\mathcal{Z}(T^*\mathcal{B})\simeq \mathrm{Sympl}_\mathcal{Z}^G(T^*\mathcal{B})\simeq B_\mathbf{W}.$$

\subsection{Moment maps}
\subsubsection{Moment maps for the $G_\mathbb{C}$-action and $G$-action on $T^*\mathcal{B}$}
For any element $x\in G_\mathbb{C}$, let $L_x$ (resp. $R_x$) denote the action of left (resp. right) multiplication by $x$ on $G_\mathbb{C}$. We will use the left action to identify $G_\mathbb{C}\times\mathfrak{g}_\mathbb{C}^*$ with $T^*G_\mathbb{C}$: 
$$\begin{array}{ccc}
G_\mathbb{C}\times\mathfrak{g}_\mathbb{C}^*&\rightarrow& T^*G_\mathbb{C}\\
(x, \xi)&\mapsto&(x, L_{x^{-1}}^*\xi)
\end{array}.$$
Using the Killing form to identify $\mathfrak{g}_\mathbb{C}^*$ with $\mathfrak{g}_\mathbb{C}$, the moment maps for the left and right $G_\mathbb{C}$-action (with respect to the holomorphic symplectic form) under the above identification are given by 
$$\begin{array}{cccc}
\mu_L:& G_\mathbb{C}\times \mathfrak{g}_\mathbb{C}&\rightarrow& \mathfrak{g}_\mathbb{C}\\
 &(x,\xi)&\mapsto& \mathrm{Ad}_x\xi
\end{array}, \text{ and }\begin{array}{cccc}
\mu_R:& G_\mathbb{C}\times \mathfrak{g}_\mathbb{C}&\rightarrow& \mathfrak{g}_\mathbb{C}\\
 &(x,\xi)&\mapsto& \xi
\end{array}, \text{ respectively.}$$

For the right Hamiltonian $B$-action on $T^*G_\mathbb{C}$ induced from the right $G_\mathbb{C}$-action,  the moment map is given by
$$\begin{array}{cccc}
\mu_{R,B}:& G_\mathbb{C}\times \mathfrak{g}_\mathbb{C}&\rightarrow& \mathfrak{b}^*\cong \mathfrak{g}_\mathbb{C}/\mathfrak{n}\\
 &(x,\xi)&\mapsto& \overline{\xi}
\end{array},$$
where $\overline{\xi}$ means the image of $\xi$ under the quotient map $\mathfrak{g}_\mathbb{C}\rightarrow\mathfrak{g}_\mathbb{C}/\mathfrak{n}$. Then we have 
$T^*\mathcal{B}=\mu_{R,B}^{-1}(0)/B=G_\mathbb{C}\times_B \mathfrak{n}$, where $B$ acts on the right on $G_\mathbb{C}$ and acts adjointly on $\mathfrak{n}$ in the last twisted product. In the following, we will also use $(x,\xi), \xi\in\mathfrak{n}$, to denote a point in $T^*\mathcal{B}$, though it should be understood as a representative in the equivalence class under the relation $(x,\xi)\sim (xb, \mathrm{Ad}_{b^{-1}}\xi)$. Now the moment map for the left $G_\mathbb{C}$-action on $T^*\mathcal{B}$ is given by
\begin{equation}\label{Springer}
\begin{array}{cccc}
\mu_\mathbb{C}:& T^*\mathcal{B}&\rightarrow& \mathfrak{g}_\mathbb{C}\\
 &(x,\xi)&\mapsto& \mathrm{Ad}_x\xi
\end{array}.
\end{equation}
Since the image of $\mu_\mathbb{C}$ is the nilpotent cone $\mathcal{N}\subset\mathfrak{g}$, we will sometimes write the codomain of $\mu_\mathbb{C}$ as $\mathcal{N}$, and then it becomes the \emph{Springer resolution}. The fiber of the Springer resolution  over $u\in \mathcal{N}$ is called a \emph{Springer fiber}, and is denoted by $\mathcal{B}_u$. Here we recall some basic facts about the Springer resolutions.

The nilpotent cone $\mathcal{N}$ is stratified by $G_\mathbb{C}$-orbits, and they form a partially ordered set. The greatest one in the poset is the open dense orbit consisting of regular nilpotent elements, and is denoted by $\mathcal{N}_{reg}$. $\mathcal{N}_{reg}$ covers a unique orbit called the \emph{subregular} orbit and is denoted by $\mathcal{N}_{sub}$. The least element in the poset is the zero orbit and it is covered by a unique orbit called the \emph{minimal orbit}, denoted by $\mathcal{N}_{min}$. For $G_\mathbb{C}=SL_n(\mathbb{C})$, $\mathcal{N}$ is the set of all nilpotent matrices, the orbits are determined by the Jordan normal form, and are classified by partitions of $n$. We will use the notation $(n_1^{k_1}, n_2^{k_2},..., n_\ell^{k_\ell})$ to denote the partition of $n$ by $k_i$ copies of $n_i$, for $i=1,...,\ell$ and $n_1>n_2>\cdots>n_\ell\geq 1$.

The Springer fibers $\mathcal{B}_u$ have irreducible components indexed by Young tableaux, and over the above mentioned orbits, the geometry is well-known: if $u\in\mathcal{N}_{reg}$, then $\mathcal{B}_u$ is a point; for $u\in \mathcal{N}_{sub}$, $\mathcal{B}_u$ is the Dynkin curve determined by the root system; for $u=0$, $\mathcal{B}_u=\mathcal{B}$; for $u\in\mathcal{N}_{min}$, if $G_\mathbb{C}=SL_n(\mathbb{C})$, then each component is a fiber bundle over the Grassmannian of $k$-planes in $\ker u$ with fiber a product of flag varieties determined by the $k$-plane, $0\leq k<n-1$. Except for some specific types, the geometry and topology of Springer fibers (mostly about their singularities) are largely unknown. The celebrated Springer correspondence gives a correspondence between the irreducible representations of the Weyl group and the Weyl group action on the top homology of the Springer fibers.

Similar formulas for the above moment maps  apply to the left and right $G$-action on $T^*G$ and $T^*\mathcal{B}\cong T^*(G/T)$, with respect to the real symplectic forms. In particular, we have the identification $T^*\mathcal{B}\cong G\times_T\mathfrak{t}^\perp$, and we will use $(x,\xi), \xi\in\mathfrak{t}^\perp$ to denote a point in $T^*\mathcal{B}$, and the moment map for the left $G$-action, after identifying $\mathfrak{g}^*$ with $i\mathfrak{g}$, is given by
\begin{equation}
\begin{array}{cccc}
\mu:& T^*\mathcal{B}&\rightarrow& i\mathfrak{g}\\
 &(x,\xi)&\mapsto& \mathrm{Ad}_{x}\xi
\end{array}.
\end{equation}

\begin{lemma}\label{sing_val}
For $G=SU(n)$, the singular values of $\mu$ in a (open) Weyl chamber of $i\mathfrak{t}$ are exactly those $p$ such that $p$ has a proper subset of eigenvalues that sum up to zero.  
\end{lemma}
\begin{proof}
Let $\overset{\circ}{W}$ be a (open) Weyl chamber in $i\mathfrak{t}$. We first show that $\mu^{-1}(\overset{\circ}{W})$ is a symplectic manifold with a Hamiltonian $T$-action, and the restriction of $\mu$ is just the moment map for the $T$-action. For this to hold, we need the composition
\begin{equation}\label{eq: sing_val}
T^*\mathcal{B}\overset{\mu}{\rightarrow} i\mathfrak{g}\rightarrow i\mathfrak{g}/i\mathfrak{t}\cong \mathfrak{t}^\perp
\end{equation}
to be a submersion along $\mu^{-1}(\overset{\circ}{W})$, and we need to specify the symplectic complement to each tangent space of $\mu^{-1}(\overset{\circ}{W})$. For any point $(x,\xi)\in T^*\mathcal{B}$, 
we have the relation $d\mu_{(x,\xi)}(L_\eta)=[\eta, \mu(x,\xi)]$. Since for any $p\in \overset{\circ}{W}$, we have $[i\mathfrak{t}^\perp, p]=\mathfrak{t}^\perp$, by the regularity of $p$ as an element in $i\mathfrak{t}$, 
we see that (\ref{eq: sing_val}) is a submersion along $\mu^{-1}(\overset{\circ}{W})$ and $\{L_\eta: \eta\in i\mathfrak{t}^\perp\}$ naturally gives a complement to the tangent spaces of $\mu^{-1}(\overset{\circ}{W})$. 
Now we just need to show that $H_\eta, \eta\in i\mathfrak{t}^\perp$ is constant on $\mu^{-1}(\overset{\circ}{W})$ and $\{L_\eta: \eta\in i\mathfrak{t}^\perp\}$ at any point $(x,\xi)\in T^*\mathcal{B}$ is a symplectic subspace of $T_{(x,\xi)}(T^*\mathcal{B})$. The first one follows from the fact that 
$$H_\eta(\mu^{-1}(\overset{\circ}{W}))=\langle\mu, \eta\rangle(\mu^{-1}(\overset{\circ}{W}))=\langle \overset{\circ}{W}, \eta\rangle=\{0\}.$$
For the second one, because of the equality
$$\omega_{(x,\xi)}(L_{\eta_1}, L_{\eta_2})=\langle [\eta_1,\eta_2], \mu(x,\xi)\rangle=\langle\eta_1,[\eta_2, \mu(x,\xi)]\rangle,$$
if $\mu(x,\xi)\in \overset{\circ}{W}$ and $\eta_2\in i\mathfrak{t}^\perp-\{0\}$, then the 1-form $\langle -,[\eta_2, \mu(x,\xi)]\rangle$ is nonzero on $i\mathfrak{t}^\perp$. 


Since there is a nontrivial center in $G$, to make the $T$-action quasi-free (i.e. the stabilizer of any point is a connected subgroup of $T$), we quotient out the center in $G$ and consider the action by the adjoint group.  Then $(x,\xi)\in \mu^{-1}(\overset{\circ}{W})$ is a singular point of $\mu$ if and only if $(x,\xi)$ has a nontrivial stabilizer by the $T$-action. This is exactly when $\xi$ has a nontrivial stabilizer in $T$. 

If $p$ has a proper subset of eigenvalues that sum up to zero, then $p=\begin{bmatrix}p_1& \\
 &p_2
\end{bmatrix}$ up to conjugation, for some $p_1\in i\mathfrak{su}(k)$ and $p_2\in i\mathfrak{su}(n-k)$, then we can find $\xi\in \mu^{-1}(p)$ of the form $\begin{bmatrix}\xi_1 & \\ &\xi_2\end{bmatrix}$ (up to conjugation) with $\xi_i\in \mu^{-1}(p_i)$, $i=1,2$. Then $\xi$ has nontrivial stabilizers containing $\mathrm{diag}(e^{i\theta},...,e^{i\theta}, e^{i\rho},..., e^{i\rho})$ with $k\theta+(n-k)\rho\in 2\mathbb{Z}\pi.$
Conversely, assume $\xi$ is fixed by an element of the form $\mathrm{diag}(e^{i\theta},...,e^{i\theta}, e^{i\theta_1},...,e^{i\theta_{n-k}})$ (up to conjugation), where the first $0<k<n$ entries are all $e^{i\theta}$, and $\theta_j-\theta\notin 2\mathbb{Z}\pi$ for $j=1,...,n-k$. Then we have $\xi_{j\ell}=0$ for $j\in\{1,...,k\}, \ell\in\{k+1,...,n\}$, thus $\xi$ is of the form $\begin{bmatrix}\xi_1&  \\  &\xi_2 \end{bmatrix}$ (up to conjugation), where $\xi_i\in \mu^{-1}(p_i), i=1,2$ for some $p_1\in i\mathfrak{su}(k)$ and $p_2\in i\mathfrak{su}(n-k)$.
\end{proof}

For any $p\in i\mathfrak{g}$, we will use $G_p$ to denote the stabilizer of $p$ in the coadjoint action by $G$, and $M_p$ to denote for the (possibly singular) reduced space $\mu^{-1}(p)/G_p$.
\begin{prop}
For any $p\in i\mathfrak{t}$, $M_p$ is naturally identified with the reduced space at zero of the $T$-action on the coadjoint orbit $\mathcal{O}(p)$. 
\end{prop}
\begin{proof}
Note that $\mathcal{O}(p)$ is the reduced space at $p$ of the left $G$-action on $T^*G$. Since the left and right $G$-actions on $T^*G$ commute, taking 2-step Hamiltonian reductions in both orders are the same. 
\end{proof}

\begin{lemma}\label{inv_moment}
Any $G$-equivariant symplectomorphism $\varphi$ of $T^*\mathcal{B}$ must preserve $\mu$, i.e. $\mu\circ \varphi=\mu.$
\end{lemma}
\begin{proof}
Since $\varphi$ is $G$-equivariant, $\mu\circ\varphi$ is also a moment map for the $G$-action on $T^*\mathcal{B}$. Note the dual of the moment map $\mathfrak{g}\rightarrow C^\infty(T^*\mathcal{B})$ is unique up to a functional $\sigma\in \mathfrak{g}^*$ such that $\sigma$ vanishes on $[\mathfrak{g},\mathfrak{g}]$ (see 5.2 in \cite{McDuff}). By semisimplicity of $\mathfrak{g}$, $[\mathfrak{g},\mathfrak{g}]=\mathfrak{g}$, so $\sigma=0$. Therefore, $\mu\circ \varphi=\mu.$

\end{proof}

\section{Construction of the surjective homomorphism $\beta_G:\mathrm{Sympl}_\mathcal{Z}^{G}(T^*\mathcal{B})\rightarrow B_{\mathbf{W}}, G=SU(n)$}
Since the moment map $\mu: T^*\mathcal{B}\rightarrow i\mathfrak{g}$ factors through
$\mu_\mathbb{C}: T^*\mathcal{B}\rightarrow\mathcal{N}$, every Spinger fiber is contained in $\mu^{-1}(p)$ for some $p$. For $G=SU(n)$, $\mu$ is the composition of $\mu_\mathbb{C}$ with the map $\mathcal{N}\rightarrow i\mathfrak{s}\mathfrak{u}(n), u\mapsto \frac{i}{2}(u-u^*)=\frac{1}{2}((iu)+(iu)^*)$. Also the two descriptions of $T^*\mathcal{B}$ by $G_{\mathbb{C}}\times_B\mathfrak{n}$ and $G\times_Ti\mathfrak{t}$ are identified by $(x,u)\mapsto (x,\frac{i}{2}(u-u^*))$, where we only choose $x\in G$. In the following, we will call a Springer fiber \emph{nontrivial} if it is not a point, and we will denote its type by the type of the nilpotent orbit it corresponds to. 

\begin{prop}
If $p\in i\mathfrak{s}\mathfrak{u}(n)$ has $n-1$ positive eigenvalues or $n-1$ negative eigenvalues, then $\mu^{-1}(p)$ does not contain any nontrivial Springer fibers. 
\end{prop}
\begin{proof}
For an element $u=[a_{ij}]\in\mathfrak{n}$,  $u$ is nonregular exactly when $a_{i,i+1}=0$ for some $i$. Then by conjugation of some permutation matrix, $\frac{i}{2}(u-u^*)$ has the $2\times 2$ submatrix on the upperleft corner to be zero. If $\frac{i}{2}(u-u^*)$ has $n-1$ positive eigenvalues or $n-1$ negative eigenvalues, then the top $2\times 2$ submatrix must have one positive eigenvalue and one negative eigenvalue. So the lemma follows. 
\end{proof}

\subsection{A study of certain loci in $\mu^{-1}(\mathrm{diag}(1,-1,0,...,0))$}
Let $p_n=\mathrm{diag}(1,-1,0,...,0)\in i\mathfrak{su}(n)$. Given $(x,\xi)\in \mu^{-1}(p_n)$, 
let $[\xi]_i$ denote for the matrix obtained by deleting the $i$-th row and column of $\xi$. 
Then $[\xi]_{n-1}$ lies in $\mathcal{O}(\epsilon p_{n-1})$ for some $\epsilon\geq 0$, by the Gelfand-Tsetlin pattern or basic facts about Hermitian matrices. Therefore, $\xi$ can be conjugated to the matrix $z_n$ in (\ref{char}) below, by a matrix $y_{n-1}\in SU(n-1)$ under the obvious embedding $SU(n-1)\hookrightarrow SU(n)$ (taking $y_{n-1}$ to $\begin{bmatrix}y_{n-1}& \\ &1\end{bmatrix}$). Now we calculate the characteristic polynomial of $z_n$ and see the possible values for $a_1,...,a_{n-1}$ in (\ref{char}).

\begin{align}\label{char}
&\det (z_n-\lambda I)=\det \begin{bmatrix}\epsilon-\lambda& & & & &a_1 \\
 &-\epsilon-\lambda& & & &a_2 \\
 & &-\lambda& & &a_3 \\
 & & &\ddots& &\vdots\\
 & & & &-\lambda&a_{n-1} \\
 \bar{a}_1&\bar{a}_2 &\bar{a}_3 &\cdots &\bar{a}_{n-1} &-\lambda
\end{bmatrix}\\
\nonumber=&\det  \begin{bmatrix}\epsilon-\lambda& & & & &a_1 \\
 &-\epsilon-\lambda& & & &a_2 \\
 & &-\lambda& & &a_3 \\
 & & &\ddots& &\vdots\\
 & & & &-\lambda&a_{n-1} \\
 0&0&0 &\cdots &0 &-\lambda+\frac{|a_1|^2}{\lambda-\epsilon}+\frac{|a_2|^2}{\lambda+\epsilon}+\frac{1}{\lambda}\sum\limits_{i=3}^{n-1}|a_i|^2
\end{bmatrix}\\
\nonumber=&(-1)^n(\lambda-\epsilon)(\lambda+\epsilon)\lambda^{n-3}(\lambda-\frac{|a_1|^2}{\lambda-\epsilon}-\frac{|a_2|^2}{\lambda+\epsilon}-\frac{1}{\lambda}\sum\limits_{i=3}^{n-1}|a_i|^2).
\end{align}

There are three cases. \\
(1) If $\epsilon\neq 0,1$, then we must have
\begin{align*}
&1-\frac{|a_1|^2}{1-\epsilon}-\frac{|a_2|^2}{1+\epsilon}-\sum\limits_{i=3}^{n-1}|a_i|^2=0,\ -1-\frac{|a_1|^2}{-1-\epsilon}-\frac{|a_2|^2}{-1+\epsilon}+\sum\limits_{i=3}^{n-1}|a_i|^2=0\\
&\frac{1}{\epsilon}|a_1|^2-\frac{1}{\epsilon}|a_2|^2=0,\ \sum\limits_{i=3}^{n-1}|a_i|^2=0.
\end{align*}
These are equivalent to 
$$|a_1|^2=|a_2|^2=\frac{1}{2}(1-\epsilon^2), a_3=\cdots=a_{n-1}=0.$$ Since we only care about $\xi$ up to the adjoint $T$-action, we can quotient out the adjoint actions by $\{\mathrm{diag}(e^{i\alpha}, e^{i\alpha},\cdots, e^{i\alpha}, e^{-i(n-1)\alpha}), \alpha\in [0,2\pi)\}$ on $z_n$, which commute with the image of $SU(n-1)$ in $SU(n)$, and assume that $a_1=\sqrt{\frac{1}{2}(1-\epsilon^2)}$ and $a_2=\sqrt{\frac{1}{2}(1-\epsilon^2)}e^{i\theta}.$\\
(2) If $\epsilon=0$, then we have
$$ \sum\limits_{i=1}^{n-1}|a_i|^2=1.$$\\
(3) If $\epsilon=1$, then 
$$a_1=\cdots =a_{n-1}=0.$$
In summary, if $\epsilon\neq 0,1$, then 
\begin{equation}\label{z_n}
z_n=\begin{bmatrix}\epsilon& & & & & \sqrt{\frac{1}{2}(1-\epsilon^2)}\\
 &-\epsilon& & & &\sqrt{\frac{1}{2}(1-\epsilon^2)}e^{i\theta}\\
 & &0& & &0 \\
 & & &\ddots& &\vdots\\
 & & & &0&0 \\
\sqrt{\frac{1}{2}(1-\epsilon^2)}&\sqrt{\frac{1}{2}(1-\epsilon^2)}e^{-i\theta}&0&\cdots &0&0
\end{bmatrix};
\end{equation}
if $\epsilon=0$, then $[\xi]_{n-1}=\mathbf{0}$ and the last column of $\xi$ has length square equal to 1; if $\epsilon=1$, then $[\xi]_{n-1}=p_{n-1}$ and the last column and row of $\xi$ are zero. In addition, if $\epsilon\neq 0,1$, by a direct calculation, we see that if $\xi=y_{n-1}z_n y_{n-1}^*$ and $y_{n-1}=[b_{jk}]_{1\leq j,k\leq n-1}$, then $|b_{j1}|=|b_{j2}|$ for all $j$. Furthermore, one can check that if $[\xi]_{n-1}$ is regular, then $(x,\xi)$ is subregular if and only if $\epsilon=1$ or $0<\epsilon<1$ and $\bar{b}_{n-1,1}+e^{i\theta}\bar{b}_{n-1,2}=0$ with $b_{n-1,1}\neq 0$, and $(x,\xi)$ is fixed by a nontrivial $S^1$-action if and only if $\epsilon=1$. 

\begin{lemma}\label{codim1}
There is a small neighborhood $\mathcal{U}$ of the union of projection of all the subregular Springer fibers in $M_{p_n}$ (inside the open set of the projection of all the regular and subregular Springer fibers), which is topologically the product of $\mathcal{Y}=\mu^{-1}(p_{n-1})^\mathrm{reg}/G_{p_{n-1}}$ with a disc $\Sigma$ on which a  subregular Springer fiber projects down to $n$ ordered points $Q_i, i=1,\cdots, n$ with a line segment connecting each pair of consecutive points. 
\end{lemma}
\begin{proof}
For any subregular element $(x,\xi)$, we have $[\xi]_i$ must be regular for some $1\leq i\leq n$. 
Let $\mu^{-1}(p_n)_i$ denote for the sublocus in $\mu^{-1}(p_n)$ where $[\xi]_i$ is regular. Then the projection $\pi_i: \mu^{-1}(p_n)_i\rightarrow \mu^{-1}(p_{n-1})^{\mathrm{reg}}/G_{p_{n-1}}, (x,\xi)\mapsto \hat{[\xi]}_{i}$ (modulo the adjoint $T$-action) is a submersion, where $\hat{[\xi]}_{i}$ means the rescaling of $[\xi]_i$ by a positive number so that it has eigenvalues $1,-1,0,\cdots,0$. By the calculations above, one gets that each fiber of $\pi_i$ quotient out by the $G_{p_n}$-action is a disc with polar coordinate $((1-\epsilon),\theta)$, where $0<\epsilon\leq 1$. The center $\epsilon=1$ corresponds to the $i$-th fixed point of an $S^1$-action on a subregular Springer fiber, and there is a ray (resp. two rays) when $i=1,n$ (resp. $i\neq 1,n$) in the disc that is the projection of (a portion of) subregular Springer fibers in $\mu^{-1}(p_n)$.

Since every subregular Springer fiber is fixed by an $S^1$-action, we know the family of subregular Springer fibers in $\mu^{-1}(p_n)$ modulo the $G_{p_n}$-action is precisely parametrized by $\mu^{-1}(p_{n-1})^\mathrm{reg}/G_{p_{n-1}}$. Now the lemma easily follows. 

\end{proof}

\subsection{Construction of $\beta_G: \mathrm{Sympl}^G_\mathcal{Z}(T^*\mathcal{B})\rightarrow B_{\mathbf{W}}$}\label{constr beta_G}

Let us continue on using the notations from Lemma \ref{codim1}. Fix a slice of the family of discs $\Sigma_0=\{y_0\}\times\Sigma$, choose two distinct points $Q_0, Q_{n+1}$ on the boundary and draw line segments connecting $Q_0$ (resp. $Q_n$) with $Q_1$ (resp. $Q_{n+1}$).

Now let $\mathcal{U}_s$ be a family of open sets in $M_{s\cdot p_n}, s>0$, which are identified under the $\mathbb{R}_+$-action, and which has $\mathcal{U}_1=\mathcal{U}$. We denote the image of $\Sigma_0\subset \mathcal{U}$ in $\mathcal{U}_s$ under the $\mathbb{R}_+$-action also by $\Sigma_0$. Given any $\varphi\in \mathrm{Sympl}_{\mathcal{Z}}^G(T^*\mathcal{B})$, we look at $\varphi_s|_{\mathcal{U}_{s}}$ as $s\rightarrow\infty$, where $\varphi_s$ is the induced automorphism on $M_{s\cdot p_n}$.  Since $\varphi$ has to preserve each Springer fiber at infinity and has to preserve the isotropy group of each point, we see that for $s$ very large,  $\varphi_s$ fixes every point on the vertical boundary $\partial\Sigma_0\times \mathcal{Y}$, and $\varphi_s(\Sigma_0)$ is contained in a neighborhood $\Sigma\times B_\epsilon(y_0)$, where $ B_\epsilon(y_0)$ is a small ball in $\mathcal{Y}$ centered at $y_0$. Note that the projection of the image of the paths $\varphi_s(\overset{\longrightarrow}{Q_iQ_{i+1}}), 0\leq i\leq n$ to $\Sigma_0$ can be isotoped to be disjoint except at the endpoints, relative to $\partial \Sigma_0$ and $\{Q_i\}_{i=1}^n$, otherwise $\varphi_s(\overline{Q_iQ_{i+1}}\times B_\epsilon(y_0))$ will intersect $\varphi_s(\overline{Q_jQ_{j+1}}\times B_\epsilon(y_0))$ for some $i\neq j$ away from $\{Q_i\}_{i=1}^n\times \mathcal{Y}$. 
Therefore, as $s$ becomes sufficiently large, the isotopy classes of the paths $\varphi_s(\overset{\longrightarrow}{Q_iQ_{i+1}}), 0\leq i\leq n$ within that neighborhood relative to the boundary $\partial(\Sigma\times B_\epsilon(y_0))$  and $\bigcup\limits_{i=1}^n\{Q_i\}\times B_\epsilon(y_0)$ corresponds to an element in $B_n$, the braid group of $n$-strands, and this gives the desired homomorphism for $G=SU(n)$:
\begin{equation}\label{beta_G}
\begin{array}{ccc}
\beta_G: \mathrm{Sympl}^G_\mathcal{Z}(T^*\mathcal{B})&\longrightarrow &B_{\mathbf{W}}
\end{array}
\end{equation}

\subsection{Fiberwise Dehn twists and the surjectivity of $\beta_G$}\label{sec.fd}
Fix a Borel subgroup $B$ in $G_\mathbb{C}$. Let $\alpha$ be a simple root, and $P^\mathbb{C}_\alpha$ be the corresponding minimal parabolic subgroup. Let $P_\alpha=P^\mathbb{C}_\alpha\cap G$. Since $T^*(G/T)\cong G\times_T\mathfrak{t}^\perp$ and $T^*(G/P_\alpha)\cong G\times_{P_\alpha}\mathfrak{p_\alpha}^\perp$ by the Killing form, we have a natural smooth fibration 
\begin{equation}\label{fib}
\xymatrix{&T^*\mathbb{P}^1\ar[r] &T^*(G/T)\ar[d]\\
& &T^*(G/P_\alpha),
}
\end{equation}
where the vertical arrow is given by the orthogonal projection 
$$p_\alpha: \mathfrak{t}^\perp\rightarrow \mathfrak{p}_\alpha^\perp.$$

\begin{lemma}
The fibration (\ref{fib}) is a symplectic fibration. 
\end{lemma}
\begin{proof}
For any smooth fibration $Y\rightarrow B$, when $Y$ is a symplectic manifold and each fiber is a symplectic submanifold, then there is a unique symplectic connection on the fibration defined by the symplectic complement to each fiber. If for any smooth curve $\gamma: [0,1]\rightarrow B$, the integration along the horizontal liftings of the tangent vector field of $\gamma$ exists for all time, then the fibration is a symplectic fibation. 

Let's check that each fiber of (\ref{fib}) is a symplectic submanifold and is isomorphic to $T^*\mathbb{P}^1$. Since the fibration is $G$-equivariant, we just need to check for the fiber over a point of the form $(e, s)\in G\times_{P_\alpha}\mathfrak{p}_\alpha^\perp$, where $e$ is the identity in $G$. The fiber is 
$$\{(p, \mathrm{Ad}_{p^{-1}}s+\xi)\in G\times_T\mathfrak{t}^\perp: p\in P_\alpha, \xi\in \ker p_\alpha\}.$$
The fiber can be canonically identified with $P_\alpha\times_T \ker p_\alpha$, the fiber at $(e,0)$, by forgetting the term $\mathrm{Ad}_{p^{-1}}s$, which preserves the respective restriction of the ambient symplectic form. Using the identity $(\mathfrak{p}_\alpha/\mathfrak{t})^*\cong \ker p_\alpha$, we see that $P_\alpha\times_T \ker p_\alpha$ is symplectically $T^*\mathbb{P}^1$.  

Lastly, for any smooth curve $\gamma: [0,1]\rightarrow T^*(G/P_\alpha)$ (without loss of generality, we assume $\gamma'(t)\neq 0, t\in[0,1]$), suppose there is a curve $\tilde{\gamma}: [0, t_1)\rightarrow T^*(G/T)$ that is an integral of the horizontal liftings of $\gamma'(t)$, but only exists up to $t_1<1$. Under the dilating  $\mathbb{R}_+$-action $\delta_s, s>0$, which preserves the symplectic fibration (up to scaling of the symplectic form) and contracts both $T^*(G/T)$ and $T^*(G/P_\alpha)$ towards their zero sections, the curve $\delta_s(\gamma)$ will be eventually very close to the zero section of $T^*(G/P_\alpha)$, and there exist $0<t_s^0<t_s^1<t_1$ for $0<s\ll 1$ such that
\begin{itemize} 
\item[(1)] $\delta_s(\tilde{\gamma})([0,t_s^0])\subset (T^*(G/T))|_{|\xi|\leq \epsilon}$ for a fixed very small $\epsilon>0$,
\item[(2)] $\delta_s(\tilde{\gamma})([0,t_s^1])\subset (T^*(G/T))|_{|\xi|\leq 1}$, and $\delta_s(\tilde{\gamma})(t_s^1)\in (T^*(G/T))|_{|\xi|=1}$,
\item[(3)] $t_s^1-t_s^0\rightarrow 0\text{ and } t_s^1\rightarrow t_1$ as $s\rightarrow 0$. 	  
\end{itemize}
Let $t_s^{M}\in [t_s^0, t_s^1]$ be a moment where the ratio $\frac{|(\delta_s(\tilde{\gamma}))'(t)|}{|(\delta_s(\gamma))'(t)|}$ reaches its maximum in $[t_s^0, t_s^1]$. Then there exists a sequence $s_n, s_n\rightarrow 0$ such that both
$$\lim\limits_{n\rightarrow \infty}\delta_s(\tilde{\gamma})(t_{s_n}^M), \lim\limits_{n\rightarrow \infty}\frac{(\delta_s(\gamma))'(t_{s_n}^M)}{|(\delta_s(\gamma))'(t_{s_n}^M)|}$$
exist. This would imply that 
$$\lim\limits_{n\rightarrow\infty} \frac{(\delta_s(\tilde{\gamma}))'(t_{s_n}^M)}{|(\delta_s(\gamma))'(t_{s_n}^M)|}$$
doesn't exist, for its length has limit $\infty$, but it should because it is the horizontal lifting of  $\lim\limits_{n\rightarrow \infty}\frac{(\delta_s(\gamma))'(t_{s_n}^M)}{|(\delta_s(\gamma))'(t_{s_n}^M)|}$. This  gives a contradiction to the existence of such a $\tilde{\gamma}$. 
\end{proof}

As before, we will use $(x,\xi), \xi\in \mathfrak{t}^\perp (\text{resp. }\mathfrak{p}_\alpha^\perp)$ to denote a point (up to equivalence relation) in $G\times_T\mathfrak{t}^\perp\cong T^*(G/T)$ (resp. $G\times_{P_\alpha}\mathfrak{p_\alpha}^\perp\cong T^*(G/P_\alpha)$). For each simple root $\alpha$ and $\xi\in\mathfrak{t}^\perp$, let $\xi_\alpha$ denote $-i(\xi-p_\alpha(\xi))$ for the projection $p_\alpha: \mathfrak{t}^\perp\rightarrow\mathfrak{p}_\alpha^\perp$. Now we can define a \emph{fiberwise Dehn twist} (the justification of the notion is included in the proof of Lemma \ref{tau_C}).
\begin{equation}\label{Dehn}
\tau_\alpha(x,\xi)=\begin{cases}&(x\exp(h(|\xi_\alpha|)\frac{\xi_\alpha}{|\xi_\alpha|}), \mathrm{Ad}_{\exp(-h(|\xi_\alpha|)\frac{\xi_\alpha}{|\xi_\alpha|})}\xi), \text{ if }\xi_\alpha\neq 0\\
&(x\exp(\frac{\pi}{2} E_\alpha),\mathrm{Ad}_{\exp(-\frac{\pi}{2} E_\alpha)}\xi),\text{ otherwise}
\end{cases},
\end{equation}
where $h:\mathbb{R}\rightarrow\mathbb{R}$ is a smooth increasing function satisfying $h(t)+h(-t)=\pi$ and $h(t)=\pi$ for $t\gg 0$, and $E_\alpha$ is any vector $v\in \mathfrak{p}_\alpha$ such that  $\exp(tv)\in T$ if and only if $t\in\mathbb{Z}\cdot\pi$. For example, if $G=SU(n)$, then $E_\alpha$ is of the form $e^{i\theta}\epsilon_{ij}-e^{-i\theta}\epsilon_{ji}$ for some $i,j$ with $i-j=1$, where $\epsilon_{ij}$ is the elementary matrix with all entries being zero except that the $(i,j)$-entry is 1. It is easy to check that $\tau_\alpha$ is well-defined, i.e. it doesn't depend on  the representative for a point in $G\times_T\mathfrak{t}^\perp$, and it preserves the fibration. The proof of the following Lemma also implies that parallel transport with respect to the canonical symplectic connection preserves $\tau_\alpha$, and in particular, $\tau_\alpha$ is smooth.

\begin{lemma}\label{tau_C}
$\tau_\alpha$ is a $G$-equivariant symplectomorphism of $T^*\mathcal{B}.$
\end{lemma}
\begin{proof}
The $G$-action is simply given by $g\cdot (x,\xi)=(gx,\xi)$ for $g\in G$, so it is clear that $\tau_\alpha$ is $G$-equivariant. Away from the locus where $\xi_\alpha=0$, we can add a parameter $t$ in all the parentheses of $\exp(\cdot)$ in (\ref{Dehn}) to get a one parameter family of diffeomorphism. Then it becomes the integral of some vector field $X$. We claim that $X$ is the Hamiltonian vector field of the Hamiltonian function $H=\widetilde{h}(|\xi_\alpha|)$, where $\widetilde{h}$ is an antiderivative of $h$, so $\tau_\alpha$ is the time-1 map of the Hamiltonian flow. To see this, we only need to check for every vertical vector $v$ in $T_{(x,\xi)}T^*\mathcal{B}$, because $X$ is $G$-equivariant and it preserves $\mu$, and this follows from the computation 
$$dH(v)=h(|\xi_\alpha|)\frac{\langle v,\xi_\alpha \rangle}{|\xi_\alpha|},\ \omega(X, v)=\langle h(|\xi_\alpha|)\frac{\xi_\alpha}{|\xi_\alpha|},v\rangle.$$
\end{proof}

\begin{lemma}
$\tau_\alpha$ is $G_\mathbb{C}$-equivariant at infinity.
\end{lemma}
\begin{proof}
Let $(x_n,\xi_n)$ be a sequence of points approaching $(x_\infty,\xi_\infty)\in T^\infty\mathcal{B}$, i.e. with appropriate choices of representatives, we have $\lim\limits_{n\rightarrow\infty}x_n=x_\infty$, $\lim\limits_{n\rightarrow\infty}|\xi_n|=\infty$ and $\lim\limits_{n\rightarrow\infty}\frac{\xi_n}{|\xi_n|}=\xi_\infty$. Here we have identified $T^\infty\mathcal{B}$ with the unit co-sphere bundle.

There are two cases. The first case is $\lim\limits_{n\rightarrow\infty}\frac{|\xi_{n,\alpha}|}{|\xi_n|}\neq 0$, then $\lim\limits_{n\rightarrow\infty}\tau_\alpha(x_n,\xi_n)=(x_\infty,\xi_\infty)$, and from here it is clear that $\tau_\alpha$ satisfies the partially compactly supported condition. The other case is $\lim\limits_{n\rightarrow\infty}\frac{|\xi_{n,\alpha}|}{|\xi_n|}= 0$. Let $\Phi$ be the set of roots, $\mathfrak{g}_\mathbb{C}=\mathfrak{h}_\mathbb{C}\oplus \bigoplus\limits_{\alpha\in\Phi}\mathfrak{g}_\alpha$ be the root space decomposition, and $\Delta$(resp. $\Delta^-$) be the set of positive(resp. negative) roots. 
Using the compact form $\mathfrak{g}$, one can define an $\mathbb{R}$-linear operator on  $\mathfrak{g}_\mathbb{C}\simeq \mathfrak{g}\otimes\mathbb{C}$ sending $X+iY$ to $(X+iY)^*:=-X+iY$. 
By standard fact, one can choose a basis for $\mathfrak{g}_\mathbb{C}$ as $\{e_\alpha\in \mathfrak{g}_\alpha, f_\alpha\in \mathfrak{g}_{-\alpha}, h_\alpha=[e_\alpha, f_\alpha]\in\mathfrak{h}_\mathbb{C}\}_{\alpha\in \Delta}$, where $(e_\alpha, f_\alpha, h_\alpha)$ forms a $\mathfrak{sl}_2$-triple and $f_\alpha=e_\alpha^*$ for all $\alpha\in\Delta$. Then $\mathfrak{t}^\perp$ is generated (over $\mathbb{R}$) by $\{-\frac{1}{2}(e_\alpha+f_\alpha), \frac{i}{2}(f_\alpha-e_\alpha)\}_{\alpha\in\Delta}$. 
 
For any $\xi\in \mathfrak{t}^\perp$, let $\xi^+$ be the portion of $-i\xi$ in $\mathfrak{n}$ under the decomposition $\mathfrak{g}_\mathbb{C}=\mathfrak{h}_\mathbb{C}+\mathfrak{n}+\mathfrak{n}^-$. Recall that $\mu_\mathbb{C}(x,\xi)=-2i\mathrm{Ad}_x(\xi^+)$. Now we need to show that 
\begin{equation}\label{asymp}
|\mu_\mathbb{C}(\tau_\alpha(x,\xi_n))-\mu_\mathbb{C}(x,\xi_n)|/|\xi_n|\rightarrow 0 \text{ as }n\rightarrow \infty,
\end{equation} for any fixed norm on $\mathfrak{g}_\mathbb{C}$. 
Given any $\alpha\in S$ (the set of simple roots) and $\beta(\neq \alpha)\in \Delta$, we have 
\begin{equation}\nonumber
(\mathrm{Ad}_{\exp (ae_\alpha-\bar{a}f_\alpha)}(be_\beta+\bar{b}f_\beta))^+=\mathrm{Ad}_{\exp (ae_\alpha-\bar{a}f_\alpha)}(be_\beta).
\end{equation}
This holds by the standard formula
$$\mathrm{Ad}_{\exp(X)}Y=Y+[X,Y]+\frac{1}{2!}[X,[X,Y]]+\frac{1}{3!}[X,[X,[X,Y]]]+\cdots,$$
and the fact that $\beta-n\alpha\notin \{0\}\cup\Delta^-$ for any $n\in\mathbb{Z}_{\geq 0}$. 
Write $$\xi_n=\xi_{n,\alpha}+\sum\limits_{\beta(\neq \alpha)\in\Delta}(b_\beta e_\beta+\bar{b}_\beta f_\beta)),$$ then 
$$\mu_\mathbb{C}(\tau_\alpha(x,\xi_n))-\mu_\mathbb{C}(x,\xi_n)=\begin{cases}&-2i\mathrm{Ad}_x(\mathrm{Ad}_{\exp(h(|\xi_{n,\alpha}|)\frac{\xi_{n,\alpha}}{|\xi_{n,\alpha}|})}\xi_{n,\alpha}^+- \xi_{n,\alpha}^+),\text{ if }\xi_{n,\alpha}\neq 0,\\
&0, \text{ otherwise},
\end{cases}$$
hence (\ref{asymp}) holds. 

\end{proof}

\begin{cor}\label{surj}
$\beta_G$ is surjective for $G=SU(n)$.
\end{cor}
\begin{proof}
First we prove for the case of $G=SU(3)$. The projection of the union of subregular Springer fibers and their image under the Weyl group action in $M_{N\cdot p_3}$ is a triangle with vertices  $Q_1=(x_1, N\begin{bmatrix}0&0&0\\ 0&0&1\\ 0&1&0\end{bmatrix})$,  $Q_2=(x_2, N\begin{bmatrix}0&0&1\\ 0&0&0\\ 1&0&0\end{bmatrix})$, and $Q_3=(x_3, N\begin{bmatrix}0&1&0\\ 1&0&0\\ 0&0&0\end{bmatrix})$. The edges are $\overline{Q_1Q_2}=\{(x, N\begin{bmatrix}0&0&a\\ 0&0&b\\ \bar{a}&\bar{b}&0\end{bmatrix}): |a|^2+|b|^2=1\}$, $\overline{Q_2Q_3}=\{(x, N\begin{bmatrix}0&a&b\\ \bar{a}&0&0\\ \bar{b}&0&0\end{bmatrix}): |a|^2+|b|^2=1\}$, and $\overline{Q_3Q_1}=\{(x, N\begin{bmatrix}0&a&0\\ \bar{a}&0&b\\ 0&\bar{b}&0\end{bmatrix}): |a|^2+|b|^2=1\}$, the first two of which are the projections of the subregular Springer fibers. 

If we fix a representative $x=\begin{bmatrix}\frac{1}{\sqrt{2}}&0&\frac{1}{\sqrt{2}}\\
\frac{\bar{a}}{\sqrt{2}}&-b&-\frac{\bar{a}}{\sqrt{2}}\\
\frac{\bar{b}}{\sqrt{2}}&a&-\frac{\bar{b}}{\sqrt{2}}
\end{bmatrix}$ for $\xi=\begin{bmatrix}0&a&b\\ \bar{a}&0&0\\ \bar{b}&0&0\end{bmatrix}$, and  $x=\begin{bmatrix}\frac{a}{\sqrt{2}}&\bar{b}&\frac{a}{\sqrt{2}}\\
\frac{b}{\sqrt{2}}&-\bar{a}&\frac{b}{\sqrt{2}}\\
\frac{1}{\sqrt{2}}&0&-\frac{1}{\sqrt{2}}
\end{bmatrix}$ for $\xi=\begin{bmatrix}0&0&a\\ 0&0&b\\ \bar{a}&\bar{b}&0\end{bmatrix}$. Then
$e^{\frac{1}{3}i\theta\mathrm{diag}(1,-2,1)}$ acts on $\xi$ by $\begin{bmatrix}0&e^{i\theta}a&b\\ \overline{e^{i\theta}a}&0&0\\ \bar{b}&0&0\end{bmatrix}$ in the first case, and $\begin{bmatrix}0&0&a\\ 0&0&e^{-i\theta}b\\ \bar{a}&\overline{e^{-i\theta}b}&0\end{bmatrix}$ in the second case. This means we can identify a neighborhood of $Q_2$ in $(\mu^{-1}(p_3+\mathbb{R}_{(-\epsilon,\epsilon)}\cdot \frac{1}{2}\mathrm{diag}(1,-2,1)))/\{e^{i\theta\mathrm{diag}(1,0,-1)}\}$ with a neighborhood of the origin in $(\mathbb{C}^2, \omega=\mathrm{Re}(\lambda dz_1\wedge dz_2))$ as $S^1$-equivariant symplectic manifolds, in which $z_2=0, z_1\leftrightarrow a$ over $\overline{Q_2Q_3}$ and $z_1=0, z_2\leftrightarrow b$ over $\overline{Q_1Q_2}$ and $\lambda$ is some complex number. 

We can find out $\lambda$ by comparing the moment maps. By a direct calculation, the matrix $\begin{bmatrix}0&z_1&\alpha\\ \bar{z}_1&0&z_2\\ \bar{\alpha}&\bar{z}_2&0\end{bmatrix}$ is conjugate to $\mathrm{diag}(1,0,-1)+\frac{t}{2}\mathrm{diag}(1,-2,1), t\in (-\epsilon,\epsilon)$ if and only if $|z_1|^2+|z_2|^2+|\alpha|^2=1+\frac{3}{4}t^2$ and $\alpha\bar{z}_1\bar{z}_2+\bar{\alpha}z_1z_2=t-\frac{t^3}{4}. $ By quotienting out the action of $\{e^{i\theta\mathrm{diag}(1,0,-1)}\}$, we can make $\alpha>0$ near $Q_2$, so the moment map is roughly given by $t\approx 2\mathrm{Re}(z_1z_2)$. The action of $e^{i\theta}(z_1, z_2)=(e^{i\theta}z_1, e^{-i\theta}z_2)$ on $(\mathbb{C}^2, \omega=\mathrm{Re}(\lambda dz_1\wedge dz_2))$ has moment map $\mathrm{Re}(i\lambda z_1z_2)$, so we should put $\lambda=-2i.$ Then the reduced space can be identified with a disc with the standard symplectic form in which the projection of $z_2=0$ (resp. $z_1=0$) maps to the negative (resp. positive) real line, and $\{z_1z_2\in i\mathbb{R}_+\}$ (resp. $\{z_1z_2\in i\mathbb{R}_-\}$) maps to the lower-half (resp. upper-half) plane\footnote{The reduced space is singular at $0$, but we can still identify it with a standard symplectic disc. This is discussed in more details in Section \ref{blup}.}.

Now let us see how $\tau_{\alpha_1}$ acts on the edges $\overline{Q_1Q_2}$ and $\overline{Q_2Q_3}$, where $\alpha_1$ is the simple root whose simple reflection corresponds to the Weyl group element that permutes the first and the second rows and columns. It is easy to see then $\tau_{\alpha_1} (\overset{\longrightarrow}{Q_1Q_2})=\overset{\longrightarrow}{Q_2Q_1}$. For $\overset{\longrightarrow}{Q_2Q_3}$, we have 
$$\tau_{\alpha_1}: N\begin{bmatrix}0&ia&b\\ -ia&0&0\\ b&0&0\end{bmatrix}\mapsto
N\begin{bmatrix}0&ia&b\cos(h(Na))\\ -ia&0&-b\sin(h(N a))\\ b\cos(h(N a))&-b\sin(h(N a))&0\end{bmatrix},$$
where we only record the $\xi$-component of the points, and $a,b$ are all nonnegative real numbers in the representatives. Note that the image never intersects the interior of  $\overline{Q_1Q_2}$ or $\overline{Q_3Q_1}$, and it intersects $\overline{Q_2Q_3}$ on the interval where $h(Na)=\pi$ and this is exactly when $b$ is sufficiently small. 
Using the same method, one can test the intersection of $\tau_{\alpha_1}(\overset{\longrightarrow}{Q_3Q_1})$ with $\overline{Q_iQ_j}, i\neq j$. From these and the fact that $-iab\sin(h(Na))\in i\mathbb{R}_-$, we can conclude that the picture for large $N$ is as in Figure \ref{braid}. One gets a similar picture of $\tau_{\alpha_2}$ for the other simple root $\alpha_2$, thus we complete the proof for $G=SU(3)$.
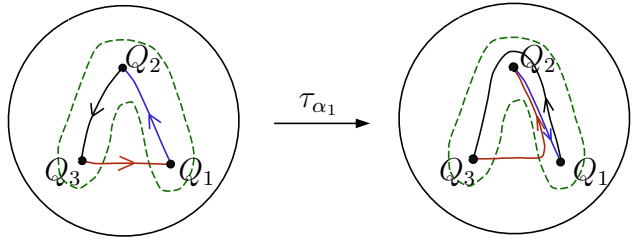
\begin{figure}
\centering
  \begin{overpic}[width=3.5in]{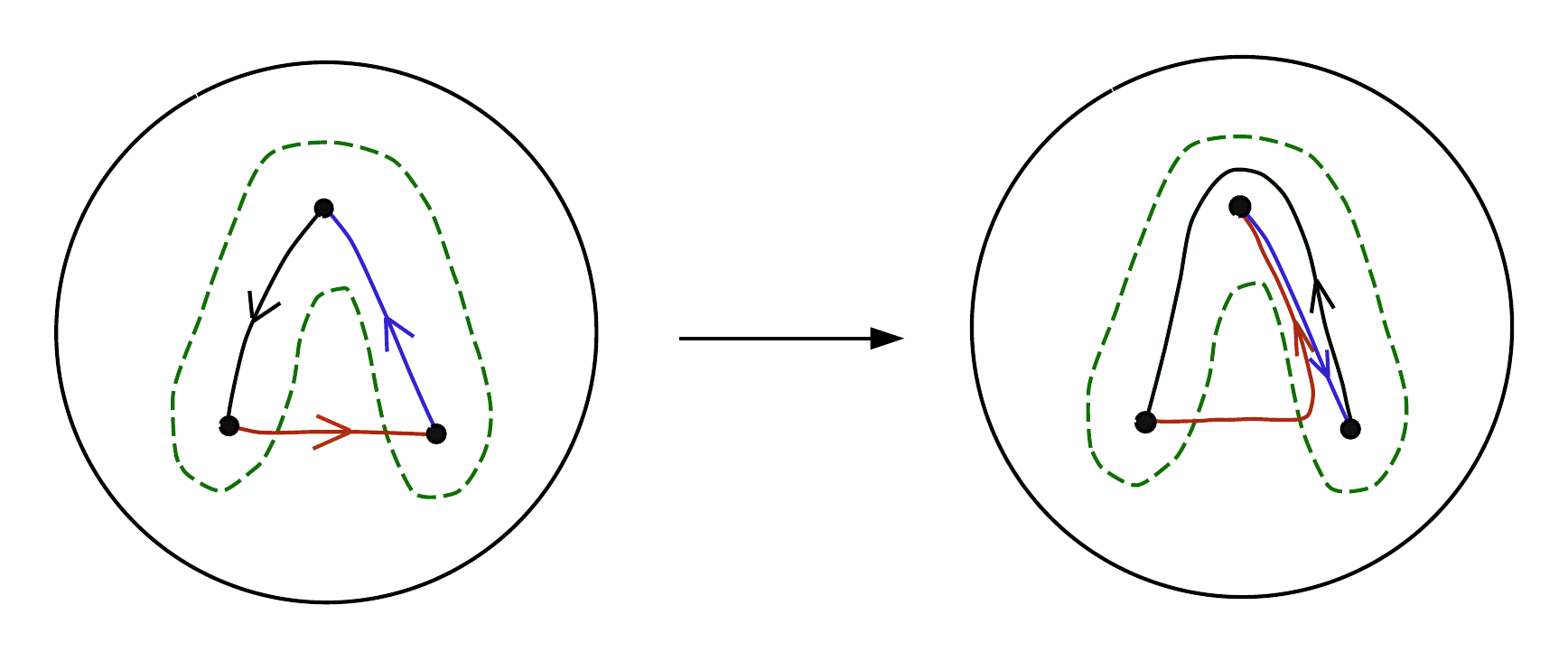}
 \put(68,12){$Q_3$}
 \put(80,29){$Q_2$}
 \put(88,12){$Q_1$}
  \put(9,12){$Q_3$}
 \put(21,29){$Q_2$}
  \put(29,12){$Q_1$}  
 \put(47,23){$\tau_{\alpha_1}$}  
\end{overpic}
\caption{The transformation of the triangle $\overline{Q_1Q_2Q_3}$ under $\tau_{\alpha_1}$.}\label{braid}
\end{figure}

For $G=SU(n)$, we look at $\varphi_s|_{\Sigma_0}$ as in Section \ref{constr beta_G}. For every three consecutive marked points $Q_i, Q_{i+1}, Q_{i+2}$, where $1\leq i\leq n-2$, we look at the submatrix consisting of the entries in the $i,(i+1), (i+2)$-th rows and columns in the $\xi$-component,  this  reduces the situation to $G=SU(3)$; see Figure \ref{glue}. It is straightforward to check that $\tau_{\alpha_i}$, for the simple root $\alpha_i$ whose simple reflection corresponds to $(i,i+1)\in S_n\cong\mathbf{W}$, reverses $\overset{\longrightarrow}{Q_iQ_{i+1}}$, keeps the isotopy classes of $\overset{\longrightarrow}{Q_jQ_{j+1}}$ for $j<i-1$ and $j>i+1$, and the image of $\overset{\longrightarrow}{Q_jQ_{j+1}}$ for $j=i-1, i+1$ is similar to that in Figure \ref{braid}. Therefore, $\{\tau_{\alpha_i}\}_{i=1}^{n-1}$ generates $B_\mathbf{W}$.
\begin{figure}
\centering
  \begin{overpic}[width=5in]{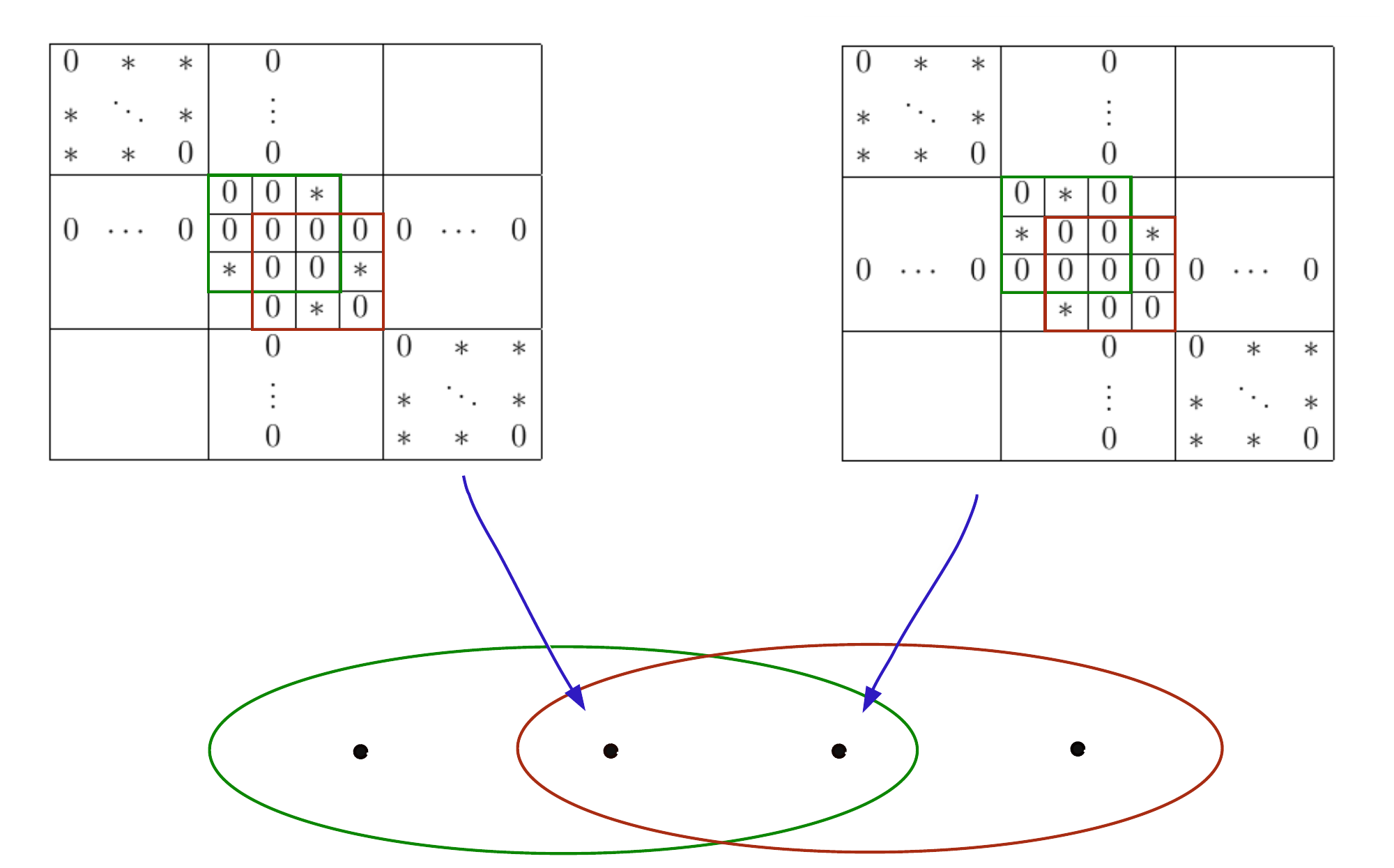}  
  \put(1,45){$i$}
  \put(19,60){$i$}
  \put(53,42){$i+1$}
  \put(77,60){$i+1$}
  \put(23,5){$Q_{i-1}$} 
  \put(43,5){$Q_i$}
  \put(57,5){$Q_{i+1}$} 
  \put(74,5){$Q_{i+2}$}
  \end{overpic}
\caption{Local picture of $\Sigma_0$.}\label{glue}
\end{figure}
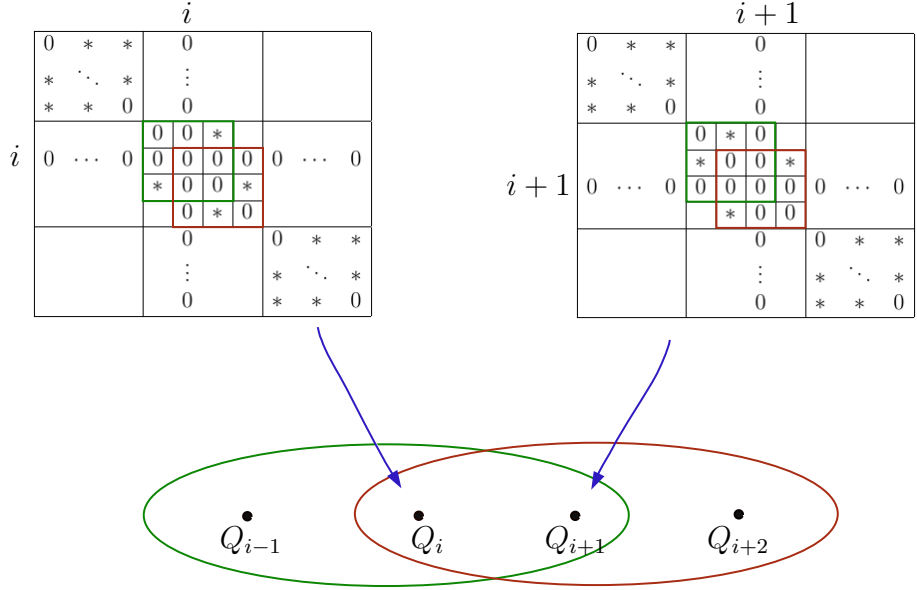

\end{proof}

\section{$\beta_G$ is a homotopy equivalence for $G=SU(3)$}
In this section, we will prove that $\mathrm{ker}\beta_G$ is contractible for $G=SU(3)$. We first review the Duistermaat-Heckman theorem and prove some basic facts for equivariant symplectomorphisms in Section \ref{D-H}. Then we divide a Weyl chamber $W$ into three parts: one around the walls, one near the singular values of $\mu$, and the other for the regular subcones. We construct symplectic local charts for their preimages under $\mu$ and trivialize the reduced spaces via the Duistermaat-Heckman theorem. We also use the technique of real blowing up to study the ``symplectomorphisms" of the reduced space over a singular value. These are done in Section \ref{sec_triv}. Lastly, in Section \ref{main_pf}, we give the proof that $\ker\beta_G$ is contractible. This is accomplished by finding the homotopy type of the symplectomorphism groups over the local charts and realizing $\ker\beta_G$ as the fiber product of these spaces. 

\subsection{Duistermaat-Heckman theorem and equivariant symplectomorphisms}\label{D-H}
Let's briefly recall the Duistermaat-Heckman theorem (c.f. \cite{GuSt89}) on the local model of the moment map near a regular value for a quasi-free Hamiltonian $T$-action. Here \emph{quasi-free} means that the stabilizer of any point is a connected subgroup of $T$.

First, the local model is the following. Let $\pi: P\rightarrow M$ be a principal $T$-bundle over a symplectic manifold $(M,\omega_0)$, with a connection form $\alpha\in\Omega^1(P,\mathfrak{t})$.  Equip $P\times \mathfrak{t}^*$ with the closed 2-form $\omega=\pi^*\omega_0+d(\tau\cdot \alpha)$, where $\pi^*\omega_0$ denotes the pull-back form under the projection $P\times\mathfrak{t}^*\rightarrow P$, and $\tau$ denotes a point in $\mathfrak{t}^*$. Since  $\omega$ is nondegenerate on $\tau=0$, there is a neighborhood $U\subset \mathfrak{t}^*$ around $0$ such that $\omega$ is a symplectic form on $P\times U$. Then the moment map for the $T$-action on $P\times U$ is given by the projection to the second factor. 

Now suppose $0$ is a regular value of a moment map $\mu: X\rightarrow \mathfrak{t}^*$ for a quasi-free Hamiltonian $T$-action on a symplectic manifold $(X,\omega_X)$. We assume that $\mu$ is proper. Then $P=\mu^{-1}(0)$ is a principal $T$-bundle over the reduced space $M_0$. Any connection form $\alpha$ on $P$ defines a trivial $T$-invariant normal bundle $F$, by $\omega: TX\overset{\sim}{\rightarrow} T^*X$. Then there is a $T$-equivariant diffeomorphism (a fiber bundle map over $U$) $\psi$ between $\mu^{-1}(U)$ and $P\times U$, for a small neighborhood $U\subset \mathfrak{t}^*$ of $0$, such that $\psi|_{P\times \{0\}}=\mathrm{id}$ and $dp_1\circ d\psi(v)=0$ for any normal vector in $F$, where $p_1: P\times U\rightarrow P$ is the projection to the first factor. Now take the above constructed $\omega$ on $P\times U$ from $\alpha$. We have $\psi^*\omega$ and $\omega_X$ agree on $P\times \{0\}$. Therefore, by the equivariant version of Moser's argument, the two manifolds are $T$-equivariantly symplectomorphic in a neighborhood of $P\times\{0\}$, and the symplectomorphism can be chosen to be the identity on  $P\times\{0\}$.

Fix a pair of dual coordinates $(t^i)_{i=1}^k$ and $(t_i)_{i=1}^k$ on $\mathfrak{t}$ and $\mathfrak{t}^*$ respectively. Let  \index{$\mathscr{C}(U, T)$}{$\mathscr{C}(U, T)$} be the subgroup of $C^\infty(U,T)$ coming from exponentiating the functions $(f_1,...,f_k)\in C^\infty(U,\mathfrak{t})$ that satisfy $\sum f_idt_i$ is an exact 1-form in $\Omega^1(U, \mathbb{R})$. In particular, it can be identified with $C_0^\infty(U, \mathbb{R})/\bigoplus\limits_{i=1}^k 2\pi\mathbb{Z}\cdot t_i$, where $C_0^\infty(U, \mathbb{R})$ means the space of all smooth functions that vanish at the origin. Note that $\mathscr{C}(U, T)$ is homotopy equivalent to $C^\infty(U,T)$, and they are identical if $T$ is of rank 1. 

\begin{prop}\label{equiv_sympl}
Assume that $H^1(M,\mathbb{R})=0$. Given a smooth family of symplectomorphisms $\{\varphi_\tau\}_{\tau\in U}$ of $M$ which preserve $F_\alpha$, there exists a $T$-equivariant symplectomorphism $\widetilde{\varphi}$ of $P\times U$ such that its induced map on the reduced space at $\tau$ is $\varphi_\tau$. The space of such $\varphi$ is a torsor over $\mathscr{C}(U, T)$.  
\end{prop}
\begin{proof}
Suppose the vector $\partial_{t_i}$ at $(x,\tau_0)$ is sent to $\partial_{t_i}+\sum\theta^i_j\partial_{t^j}+\widetilde{X}_i$, where $\widetilde{X}_i$ is the horizontal lifting of $X_i=\varphi_{\tau_0*}(\frac{d}{dt_i}|_{t_i=0}\varphi_{\tau_0}^{-1}\varphi_{\tau_0+t_i}\pi(x))$, and any horizonal lifting $\widetilde{X}$ of $X\in TM$ is sent to $\widetilde{\varphi_{\tau_0*}X}+\sum\theta^X_i\partial_{t^i}$.The condition for $\widetilde{\varphi}$ to preserve the symplectic form is that it preserves the symplectic pairing of $\partial_{t_i}, \partial_{t_j}$ and that of $\partial_{t_i}, X$ for each $i,j$ and $X$. This is the same as saying the followings
\begin{equation}\label{equiv_1}
-\theta_i^j+\theta^i_j+\omega_{\tau_0}(X_i,X_j)=0.
\end{equation}
\begin{equation}\label{equiv_2}
-\theta^X_i+\omega_{\tau_0}(X_i, X)=0.
\end{equation}
Since $X_i$ preserves $\omega_{\tau_0}$ and $H^1(M,\mathbb{R})=0$, we have  $i_{X_i}\omega_{\tau_0}=dH_{\tau_0,i}$ for a Hamiltonian function $H_{\tau_0,i}$. So (\ref{equiv_2}) is the same as 
\begin{equation}\label{equiv_3}
\widetilde{\varphi}_{\tau_0}^*\alpha-\alpha=\pi^*d(H_{\tau_0,i})_{i=1}^k\in \Omega^1(P, \mathfrak{t}),
\end{equation}
which can be easily satisfied by composing a gauge transformation with any lifting $\widetilde{\varphi}_{\tau_0}$ we started with. Now we start from any $\widetilde{\varphi}$ satisfying (\ref{equiv_3}), and we have $$\widetilde{\varphi}^*\omega-\omega=\sum\limits_{i<j}f_{ij}dt_i\wedge dt_j=d(\sum\limits_{i=1}^k g_i(\tau)dt_i).$$ 
Applying Moser's argument for $\omega_s=(1-s)\cdot\omega+s\cdot\widetilde{\varphi}^*\omega$ and the primitive  $\sigma_s=\sum\limits_{i=1}^k g_i(\tau)dt_i$, we get one desired $\widetilde{\varphi}$. 

If $\{\varphi_\tau\}_{\tau\in U}=\{id\}_{\tau\in U}$, then (\ref{equiv_1}) and $(\ref{equiv_2})$ imply that $\widetilde{\varphi}\in C^\infty(U, T)$, and a lifting of it to $(f_1,...,f_k)\in C^\infty(U,\mathfrak{t})$ satisfies that $\sum\limits_{i=1}^k f_idt_i$ is exact. So the collection of $\widetilde{\varphi}$ is exactly $\mathscr{C}(U, T)$. 
\end{proof}

\subsection{Trivialization of the reduced spaces over a Weyl chamber}\label{sec_triv}

Now we focus on $G=SU(3)$. Let \index{$w_0$, $w_1$, $w_2$}
$$w_0=\mathrm{diag}(1,1,-2), w_1=\mathrm{diag}(1,0,-1), w_2=\mathrm{diag}(2,-1,-1).$$ Let $W$ be the Weyl chamber in $\mathfrak{t}^*\cong i\mathfrak{t}$ bounded by the rays $\mathbb{R}_{\geq 0}\cdot w_0$ and $\mathbb{R}_{\geq 0}\cdot w_2$. Also, let \index{$W$, $W_{ij}$}$W_{ij}$ denote the subcone of $W$ bounded by $\mathbb{R}_{\geq 0}\cdot w_i$ and $\mathbb{R}_{\geq 0}\cdot w_j$ for $(i,j)=(0,1)$ and $(1,2)$. For any $p\in i\mathfrak{t}$, we will denote the reduced space by $M_p$, and we will use $\varphi_p$ to denote the induced map on $M_p$ by any $\varphi\in \mathrm{Sympl}_\mathcal{Z}^G(T^*\mathcal{B})$.

The action by $T$ is not quasi-free, since the center in $SU(3)$ fixes every point. This can be resolved by replacing $G=SU(3)$ by $G_{\mathrm{Ad}}=PSU(3)=SU(3)/\mu_3$, where \index{$\mu_3$}{$\mu_3$} is the center. 

\subsubsection{Trivialization around the ray $\mathbb{R}_{\geq 0}\cdot w_0$}
Fix a $p\in \mathbb{R}_{>0}\cdot w_0$. Then the Lie algebra of $G_p$ is $\mathfrak{g}_p=\{x\in\mathfrak{g}: [x,p]=0\}\cong\mathfrak{u}(2)$ (we fix such an identification once for all).

\begin{lemma}\label{w_0}
For $\epsilon>0$ small, $\mu^{-1}(p+B_\epsilon(0, i\mathfrak{g}_p))$ is $U(2)$-equivariantly symplectomorphic to a neighborhood of the zero section of  $T^*(U(2)/\mu_3)$, where $\mu_3=\{\begin{bmatrix}e^{i\frac{2k\pi}{3}}& \\ &e^{i\frac{2k\pi}{3}}\end{bmatrix}: 0\leq k\leq 2\}$. 
\end{lemma}
\begin{proof}
 First, we show that $\mu^{-1}(p+B_\epsilon(0, i\mathfrak{g}_p))$ is a symplectic submanifold with symplectic complement at each point $(x,\xi)$ consisting of the Hamiltonian vector fields $L_\eta, \eta\in i\mathfrak{g}_p^\perp$. The proof is very similar to the first part of the proof of Lemma \ref{sing_val}.
 
The map $T^*\mathcal{B}\rightarrow i\mathfrak{g}\rightarrow i\mathfrak{g}/i\mathfrak{g}_p\cong\mathfrak{g}_p^\perp$ is a submersion restricted to $\mu^{-1}(p+B_\epsilon(0, i\mathfrak{g}_p))$, for $\epsilon>0$ small enough. This is because $d\mu_{(x,\xi)} (L_{\eta})=[\eta, \mu(x,\xi)]$ for any $\eta\in\mathfrak{g}$, and  $[i\mathfrak{g}_p^\perp, p]=\mathfrak{g}_p^\perp$. Therefore, $\mu^{-1}(p+B_\epsilon(0, i\mathfrak{g}_p))$ is a smooth submanifold and $\{L_\eta(x,\xi): \eta\in i\mathfrak{g}_p^\perp\}$ is a complement to its tangent space at any point $(x,\xi)$. The tangent space of $\mu^{-1}(p+B_\epsilon(0, i\mathfrak{g}_p))$ at any point $(x,\xi)$ is spanned by $L_{\eta}, \eta\in\mathfrak{g}_p$ and the vertical vectors $L_{x^{-1}}^*\zeta, \zeta\in \mathrm{Ad}_{x^{-1}}i\mathfrak{g}_p$, 
 so clearly $\omega$ is nondegenerate on $\mu^{-1}(p+B_\epsilon(0, i\mathfrak{g}_p))$. Also, $i_{L_\eta}\omega=dH_{\eta}, \eta\in i\mathfrak{g}_p^\perp$ vanishes on $\mu^{-1}(p+B_\epsilon(0, i\mathfrak{g}_p))$,
so the space of Hamiltonian vector fields $L_\eta, \eta\in i\mathfrak{g}_p^\perp$ at each point is its symplectic complement.

Next, since $\mu^{-1}(p)\cong U(2)/\mu_3$ is $\omega$-isotropic, by the equivariant version of Weinstein's Lagrangian tubular neighborhood theorem, we get the desired result. 
\end{proof}

Let $w_\epsilon=w_0+\epsilon\cdot \mathrm{diag}(1,-1,0)$, and \index{$W_{\pm\epsilon}$}$W_{\pm\epsilon}\subset i\mathfrak{t}$ be the cone bounded by $\mathbb{R}_{\geq 0}\cdot w_{\epsilon}$ and $\mathbb{R}_{\geq 0}\cdot w_{-\epsilon}$ for $\epsilon>0$ small. Identifying $\mathrm{Ad}_{G_p}(\overset{\circ}{W}_{\pm\epsilon})$ with a cone in $i\mathfrak{u}(2)$ via the map $\mathfrak{g}_p\cong \mathfrak{u}(2)$, we have 
\begin{cor}\label{triv_bdry}
$\mu^{-1}(\mathrm{Ad}_{G_p}(\overset{\circ}{W}_{\pm\epsilon}))\cong U(2)/\mu_3\times \mathrm{Ad}_{G_p}(\overset{\circ}{W}_{\pm\epsilon})$ as Hamiltonian $U(2)$-spaces, where
the latter space is equipped with the symplectic form induced from $T^*(U(2)/\mu_3)\cong U(2)/\mu_3\times i\mathfrak{u}(2)$.
\end{cor}
\begin{proof}
The symplectic form on $U(2)/\mu_3\times i\mathfrak{u}(2)$ is invariant under the translation map $(\cdot, \cdot+v)$ for any $v\in\mathbb{R}\cdot \mathrm{diag}(1,1)$, and it is getting scaled under the $\mathbb{R}_+$-action on $i\mathfrak{u}(2)$. Note that such change of the symplectic form is compatible with the $\mathbb{R}_+$-action on $\mu^{-1}(\mathrm{Ad}_{G_p}(\overset{\circ}{W}_{\pm\epsilon}))$, so combining with Lemma \ref{w_0}, we complete the proof.
\end{proof}

\subsubsection{Trivialization along $\mathring{W}$ and $\mathring{W}_{ij}$, $(i,j)=(0,1)$ and $(1,2)$}
The following lemma has already been obtained within the proof of Lemma \ref{sing_val}. 
\begin{lemma}\label{W_sympl}
$\mu^{-1}(\overset{\circ}{W})$ is a symplectic submanifold with symplectic complement consisting of the tangent vectors to the $\exp(B_\epsilon(0,-i\mathfrak{t}^\perp))$-orbits. In particular, the same holds for $\mu^{-1}(\overset{\circ}{W_{ij}})$,  for $(i,j)=(0,1)$ and $(1,2)$.
\end{lemma}

For any $p\in\overset{\circ}{W}_{01}$, $\mu^{-1}(p)\cong U(2)/\mu_3$ as a principal $T$-bundle over $\mathbb{P}^1(\cong T\backslash U(2)$, the quotient of $U(2)$ by the left action of $T$). Let $A\in\Omega^1(U(2)/\mu_3, i\mathfrak{t})$ be the unique right $U(2)$-invariant connection form on $U(2)/\mu_3$ determined by the Killing form, i.e. one takes the Maurer-Cartan form and projects it to $i\mathfrak{t}$. Applying Duistermaat-Heckman theorem (see \cite{GuSt89}), we get the following.
\begin{prop}\label{W_{01}}
$\mu^{-1}(\overset{\circ}{W}_{01})$ is $T$-equivariantly symplectomorphic to $U(2)/\mu_3\times \overset{\circ}{W}_{01}$ with symplectic form $c\cdot d(A\cdot \tau)$, where $\tau\in\mathfrak{t}^*$ and $c$ is some positive constant. The symplectomorphism can be chosen to respect the $\mathbb{R}_+$-action.
\end{prop}
\begin{proof}
The only thing to be careful is that we have a global identification over $\overset{\circ}{W}_{01}$ rather than a local identification near some point. First, on each reduced space, the cohomology class of $cdA\cdot \tau$ agrees with that of the induced symplectic form, for some fixed $c>0$. This is because the latter depends linearly on $\tau$ and the class vanishes on $\mathbb{R}_{\geq 0}\cdot w_0$.

Fix a $T$-equivariant isomorphism $\phi: \mu^{-1}(\overset{\circ}{W}_{01})\rightarrow U(2)/\mu_3\times \overset{\circ}{W}_{01}$. The fact that the reduced spaces are all $\mathbb{P}^1$ ensures that we can apply the equivariant version of Moser's argument on the family of symplectic forms $(1-t)\omega|_{\mu^{-1}(\overset{\circ}{W}_{01})}+t \phi^* d(A\cdot\tau), t\in [0,1]$, and get the statement.
\end{proof}

\subsubsection{Real blow-ups and some treatment near the singular loci}\label{blup}
$\empty$\\

\noindent\ref{blup}.1. \emph{Real blowing up operations and local charts near the singular loci of $\mu$.}
The material below on real blow-ups is following \cite{GuSt89}, section 10. Suppose we have a Hamiltonian $S^1$-action on $\mathbb{C}\times\mathbb{C}^n$ (equipped with the product of the standard K$\ddot{\text{a}}$hler forms), given by 
\begin{equation}\label{S^1-action}
e^{i\theta}\cdot (z_0, z)=(e^{i\theta}z_0, e^{-i\theta}z).
\end{equation}  Then the moment map is $\Phi(z_0, z)=-|z_0|^2+|z|^2$. The real blowing up is a local surgery to $\mathbb{C}\times\mathbb{C}^n$, so that $\Phi^{-1}(-\infty, 0)$ is unchanged and the new moment map is regular over $(-\infty, \delta)$ for some $\delta>0$. The construction is as follows. 

Let $(t,s)$ be the standard coordinate on $T^*S^1\cong S^1\times\mathbb{R}$.  Choose $\epsilon,\delta>0$ very small, remove the set $\{|z_0|^2<\frac{\epsilon}{2}, -|z_0|^2+|z|^2<\delta\}$ in $\mathbb{C}\times\mathbb{C}^n$ and glue with the set $\{s<\epsilon, -s+|z|^2<\delta\}\subset T^*S^1\times\mathbb{C}^n$ using the identification $\{\frac{\epsilon}{2}\leq |z_0|^2<\epsilon, -|z_0|^2+|z|^2<\delta\}\cong \{\frac{\epsilon}{2}\leq s<\epsilon, -s+|z|^2<\delta\}$. We will denote the resulting manifold by $Bl_{\epsilon, \delta}(\mathbb{C}\times\mathbb{C}^n)$. Since the real blowing up can be done within an arbitrarily small ball around the origin for $\epsilon,\delta$ sufficiently small, we can globalize this procedure to any quasi-free Hamiltonian $S^1$-action on a symplectic manifold $M$ with the moment map having Morse-Bott singularities of index $(2, 2k)$.

Now let \index{$T^{\check{w}}$} $T^{\check{w}}$ denote the subgroup $\exp(\mathbb{R}\cdot \check{w})$ for any $\check{w}\in \mathfrak{t}$. As mentioned before, we have to replace $G$ by $G_{\mathrm{Ad}}$ to ensure the action by $T$ to be quasi-free. Let
 \index{$u_1$, $\check{u}_1$, $\underline{w}_1$, $\check{\underline{w}}_1$}$$u_1=\frac{1}{2}\mathrm{diag}(1,-2,1), \check{u}_1=\frac{1}{3}\mathrm{diag}(i,-2i, i),\underline{w}_1=\frac{1}{2}w_1, \check{\underline{w}}_1=iw_1.$$
 For $\nu>0$ small, let \index{$C_\nu$}$C_\nu$ be the cone bounded by $\mathbb{R}_{\geq 0}(\underline{w}_1\pm \nu\cdot u_1)$. It is clear that  $\mu^{-1}(\overset{\circ}{C}_\nu)$ can be trivialized as $\mu^{-1}(\underline{w}_1+\mathbb{R}_{(-\nu,\nu)}\cdot u_1)\times \mathbb{R}_+$ equipped with the symplectic form $d(t\alpha)$, where $\alpha$ is equal to the primitive $-\mathbf{p}d\mathbf{q}$ of $\omega$ to $\mu^{-1}(\underline{w}_1+\mathbb{R}_{(-\nu,\nu)}\cdot u_1)$ and $t$ is the coordinate of $\mathbb{R}_+$. Along $\mu^{-1}(\overset{\circ}{C_\nu})$, $T^{\check{\underline{w}}_1}$ acts freely, so the moment map 
$$\mu_{\underline{w}_1,\nu}: \mu^{-1}(\overset{\circ}{C_\nu})\overset{\mu}{\rightarrow}\overset{\circ}{C_\nu}\rightarrow \overset{\circ}{C_\nu}/\langle u_1\rangle\cong\mathbb{R}_+\cdot \underline{w}_1$$
for the $T^{\check{\underline{w}}_1}$-action is regular, and the reduced space at any $p\in  \mathbb{R}_+\cdot \underline{w}_1$ \index{$M_{p,\nu}^{\underline{w}_1}$, $\mu_{p,\nu}^{\underline{w}_1}$}
$$M_{p,\nu}^{\underline{w}_1}:=T^{\check{\underline{w}}_1}\backslash\mu_{\underline{w}_1,\nu}^{-1}(p)$$
is a 4-dimensional symplectic manifold with a Hamiltonian $T^{\check{u}_1}$-action. 
The moment map for the $T^{\check{u}_1}$-action on $M_{p,\nu}^{\underline{w}_1}$ is denoted by 
$$\mu_{p,\nu}^{\underline{w}_1}: M_{p,\nu}^{\underline{w}_1}\rightarrow (\mathbb{R}\cdot \check{u}_1)^*\cong\mathbb{R}\cdot u_1.$$ 

By Lemma \ref{codim1}, $T^{\check{u}_1}$ has exactly three fixed points $Q_j, j=1,2,3$ on $M_{p,\nu}^{\underline{w}_1}$, which are of the form $(x_1, \begin{bmatrix}0&0&0\\ 0&0&1\\0&1&0\end{bmatrix})$, $(x_2, \begin{bmatrix}0&0&1\\ 0&0&0\\1&0&0\end{bmatrix})$ and $(x_3, \begin{bmatrix}0&1&0\\ 1&0&0\\0&0&0\end{bmatrix})$ respectively, when $p=w_1$. 
Since $(\mathbb{C}^2, \omega=\mathrm{Re}(-idv_1\wedge dv_2))$ and $(\mathbb{C}^2, \omega_{\text{st}}=\frac{i}{2}(dz_0\wedge d\overline{z}_0+dz_1\wedge d\overline{z}_1)))$ are related by $z_0=v_1+\overline{v}_2, z_1=-\overline{v}_1+v_2$, by the calculation of Corollary \ref{surj}, we can identify a small neighborhood of each $Q_j$ in $M_{p,\nu}^{\underline{w}_1}$ with $(\mathbb{C}^2, \omega_{\text{st}}=\frac{i}{2}(dz_0\wedge d\overline{z}_0+dz_1\wedge d\overline{z}_1))$ in an $S^1$-equivariant way. The reduced space at 0 for the Hamiltonian action in (\ref{S^1-action}) can be identified with $\mathbb{C}$ (with the standard K$\ddot{\text{a}}$hler form) by taking the slice in $\{|z_0|^2=|z_1|^2\}$ in which $z_0\geq 0$ and $z_1$ is used to be the linear coordinate on $\mathbb{C}$. In particular, under such identifications, we have $\overline{Q_{i-1}Q_{i}}$ and $\overline{Q_iQ_{i+1}}$ in Figure \ref{braid} go to the positive and negative real lines respectively near $Q_i$, where the indices $i$ are taken to be modulo 3.

Now we can desingularize the action by $T^{\check{u}_1}$ along $\mathbb{R}_{>0}\cdot \underline{w}_1$,   and replace $\mu|_{\mu^{-1}(\overset{\circ}{W})}$ by $\widetilde{\mu}$, then $\widetilde{\mu}$ is regular over the interior of the cone $W_{01,\delta}$ bounded by $\mathbb{R}_{\geq 0}\cdot w_0$ and $\mathbb{R}_{\geq 0}\cdot (\underline{w}_1-\delta\cdot u_1)$, for some $\delta>0$. Similarly to Proposition \ref{W_{01}}, we have $\widetilde{\mu}^{-1}(\overset{\circ}{W}_{01,\delta})\cong (U(2)/\mu_3\times \overset{\circ}{W}_{01,\delta}, c\cdot d(A\cdot \tau))$. 

\begin{remark}\label{rem_triv}
Since the blowing down map from $Bl_{\epsilon, \delta}(\mathbb{C}\times\mathbb{C}^n)$ to $\mathbb{C}\times \mathbb{C}^n$ identifies the reduced spaces at 0, this gives a way to identify the reduced spaces over $\mathbb{R}_{>0}\cdot \underline{w}_1$ with the others over $\overset{\circ}{W}_{01}$. 
\end{remark}

\noindent\ref{blup}.2. \emph{The equivariant linear Symplectic group $Sp(4)^{S^1}$.}

Let \index{$Sp(4)^{S^1}$}
$$Sp(4)^{S^1}:=\{P\in Sp(4): P\text{ commutes with the $S^1$-action in (\ref{S^1-action})}\},$$
where $P$ is relative to the standard basis $\partial_{x_0},\partial_{y_0},\partial_{x_1},\partial_{y_1}$ and $z_j=x_j+iy_j$ for $j=0,1$.
\begin{lemma}\label{l.dvarphi}
$$Sp(4)^{S^1}=\{\begin{bmatrix}\lambda_1e^{i\theta_1}&\lambda_2\sigma\circ e^{i\theta_2}\\
\lambda_2\sigma\circ e^{i\theta_3}& \lambda_1e^{i\theta_4}
\end{bmatrix}, \lambda_i\geq 0\text{ for }i=1,2, \lambda_1^2-\lambda_2^2=1, \theta_1+\theta_2=\theta_3+\theta_4 \text{ if }\lambda_2\neq 0\}, $$
where $\sigma$ means taking complex conjugate.
\end{lemma}
\begin{proof}
Let $P=\begin{bmatrix}A&B\\ C&D
\end{bmatrix}$, where $A,B,C,D$ are all $2\times 2$-matrices. Let $R_\theta$ denote the standard rotation matrix on $\mathbb{R}^2$ by angle $\theta$. Then $P$ is $S^1$-equivariant implies that 
$$[\begin{bmatrix}A&B\\ C&D
\end{bmatrix}, \begin{bmatrix}R_{-\theta}&0\\
0&R_{\theta}
\end{bmatrix}]=0, $$
and this is equivalent to that $P$ is of the form
\begin{equation}\label{dvarphi}
\begin{bmatrix}\lambda_1e^{i\theta_1}&\lambda_2\sigma\circ e^{i\theta_2}\\
\lambda_3\sigma\circ e^{i\theta_3}& \lambda_4e^{i\theta_4},
\end{bmatrix}, \lambda_i\geq 0\text{ for }i=1,...,4
\end{equation}
relative to the standard basis $\partial_{z_0}, \partial_{z_1}$.
Now we need $P$ to be symplectic, i.e. it preserves the K$\ddot{\text{a}}$hler form 
$\frac{i}{2}(dz_0\wedge d\bar{z}_0+dz_1\wedge d\bar{z}_1)$. By direct calculations, 
the undetermined quantities in (\ref{dvarphi}) should satisfy
$$\lambda_1=\lambda_4, \lambda_2=\lambda_3, \lambda_1^2-\lambda_2^2=1, \text{ and }\theta_1+\theta_2=\theta_3+\theta_4\text{ if }\lambda_2\neq 0,$$
and this completes the proof.
\end{proof}
\index{$C_0$}Let $C_0$ be the center of $Sp(4)^{S^1}$, i.e. $\{\begin{bmatrix}e^{-i\theta}& \\ &e^{i\theta}
\end{bmatrix}, \theta\in [0,2\pi)\}$. By the above Lemma, 
\begin{equation}\label{Sp/S1}
Sp(4)^{S^1}/C_0\cong \{\begin{bmatrix}\lambda_1&\lambda_2\sigma\circ e^{i\theta_2}\\
\lambda_2\sigma\circ e^{i\theta_3}& \lambda_1e^{i\theta_4}
\end{bmatrix}: \lambda_1^2-\lambda_2^2=1, \lambda_i\geq 0\text{ for }i=1,2; \theta_2=\theta_3+\theta_4 \text{ if }\lambda_2\neq 0\}.
\end{equation}
There is an  $S^1$-action by the left multiplication of the subgroup $\{\begin{bmatrix}1& \\
 &e^{i\theta}
\end{bmatrix}\}$, and the projection 
$$\begin{array}{ccc}
Sp(4)^{S^1}/C_0&\longrightarrow& S^1\\
\begin{bmatrix}\lambda_1&\lambda_2\sigma\circ e^{i\theta_2}\\
\lambda_2\sigma\circ e^{i\theta_3}& \lambda_1e^{i\theta_4}
\end{bmatrix}&\mapsto&\{\begin{bmatrix}1& \\
 &e^{i\theta_4}
\end{bmatrix}\}
\end{array}
 $$
is an $S^1$-equivariant fiber bundle, with each fiber homeomorphic to a disc, so in particular, this map is a homotopy equivalence.

\begin{lemma}\label{angle}
Let $z_+(t)=t$ and $z_-(t)=-t$, $t\geq 0$ be the two opposite rays emitting from the origin in the reduced space $\mathbb{C}$ at 0. Let $P\in Sp(4)^{S^1}$ and $\widetilde{P}$ be the induced map on $\mathbb{C}$. Then \\
(a) there exists a $P$ for any prescribed values of $\arg(\frac{d}{dt}|_{t=0}\widetilde{P}(z_+(t)))$ and $\arg(\frac{d}{dt}|_{t=0}\widetilde{P}(z_-(t)))$,  except for $\arg(\frac{d}{dt}|_{t=0}\widetilde{P}(z_+(t)))=\arg(\frac{d}{dt}|_{t=0}\widetilde{P}(z_-(t)))$.\\
(b)  If $\widetilde{P}\in Sp(2)$, then $P=\begin{bmatrix}1& \\ &e^{i\theta_4}
\end{bmatrix}$ for some $\theta_4$ modulo the center $C$. In particular, if $P$ satisfies 
\begin{equation}\label{fix1}
\frac{d}{dt}|_{t=0}\widetilde{P}(z_+(t))=1, \frac{d}{dt}|_{t=0}\widetilde{P}(z_-(t))=-1,
\end{equation}
then $P\in C$. Therefore, the map
\begin{equation}\label{direction}
\begin{array}{ccc}
Sp(4)^{S^1}/C_0&\longrightarrow&\mathbb{C}\times \mathbb{C}\\
P&\mapsto&(\frac{d}{dt}|_{t=0}\widetilde{P}(z_+(t)),  \frac{d}{dt}|_{t=0}\widetilde{P}(z_-(t)))
\end{array}
\end{equation}
is an injection.\\
(c) the map 
\begin{equation}\label{S^1-dir}
\begin{array}{ccc}Sp(4)^{S^1}/C_0&\longrightarrow &S^1\\
P&\mapsto& \exp(i\arg(\frac{d}{dt}|_{t=0}\widetilde{P}(z_+(t))))
\end{array}
\end{equation}
is a homotopy equivalence of spaces.
\end{lemma}
\begin{proof}
(a) One lifting of the tangent vector at $0$ of the two rays $z=t$ and $z=-t$ to $\mathbb{C}^2$ is $v_+=\begin{bmatrix}1\\1
\end{bmatrix}$ and $v_-=\begin{bmatrix}1\\-1
\end{bmatrix}$, respectively. Take $P$ as in (\ref{Sp/S1}), then 
$$Pv_+=\begin{bmatrix}\lambda_1+\lambda_2e^{-i\theta_2}\\
(\lambda_1+\lambda_2e^{-i\theta_2})e^{i\theta_4}
\end{bmatrix},Pv_-=\begin{bmatrix}\lambda_1-\lambda_2 e^{-i\theta_2}\\
-(\lambda_1-\lambda_2 e^{-i\theta_2})e^{i\theta_4}
\end{bmatrix}.$$
Let $\beta_\pm=\arg(\lambda_1\pm\lambda_2e^{-i\theta_2})$, and $w_{\pm}$ denote $\frac{d}{dt}|_{t=0}\widetilde{P}(z_{\pm}(t))$. Then
$$\arg(w_+)=\theta_4+2\beta_+, \arg(w_-)=\theta_4+2\beta_-+\pi.$$
It is not hard to see that $\beta_+-\beta_-$ ranges in $(-\frac{\pi}{2},\frac{\pi}{2})$, and then we can use $\theta_4$ to adjust $\arg(w_{\pm})$ to the prescribed values.\\
(b) The claim follows by direction calculations. 
(c) It is obvious by looking at the image of the subgroup $\{\begin{bmatrix}1&0\\0&e^{i\theta}\end{bmatrix}, \theta\in[0,2\pi)\}.$
\end{proof}

\noindent\ref{blup}.3. \emph{A deformation retraction of the equivariant symplectomorphism group of $M_{p,\nu}^{\underline{w}_1}$ to a point.}
Let \index{$\mathrm{Sympl}^{T^{\check{u}_1}}(M_{p,\nu}^{\underline{w}_1},\{Q_j\}_{j=1}^3)$} $\mathrm{Sympl}^{T^{\check{u}_1}}(M_{p,\nu}^{\underline{w}_1},\{Q_j\}_{j=1}^3)$ denote the subgroup of $T^{\check{u}_1}$-equivariant symplectomorphisms $\mathrm{Sympl}^{T^{\check{u}_1}}(M_{p,\nu}^{\underline{w}_1})$ of $M_{p,\nu}^{\underline{w}_1}$ that \emph{fix each }$Q_j, j=1,2,3$. Let \index{$\widetilde{\mathrm{Sympl}}(M_p, \{Q_j\}_{j=1}^3)$, $\widetilde{\mathrm{Sympl}}_0(M_p, \{Q_j\}_{j=1}^3)$, $\mathrm{Sympl}_0(M_p, \{Q_j\}_{j=1}^3)$} $\widetilde{\mathrm{Sympl}}(M_p, \{Q_j\}_{j=1}^3)$ denote the group of automorphisms of the reduced space $M_p$  induced from $\mathrm{Sympl}^{T^{\check{u}_1}}(M_{p,\nu}^{\underline{w}_1},\{Q_j\}_{j=1}^3)$. For each $Q_j, j=1,2,3$, we fix an identification between a neighborhood of 0 in $\mathbb{C}^2$ with a neighborhood of $Q_j$ in $M_{p,\nu}^{\underline{w}_1}$ as in Section \ref{blup}.1, and this induces an identification between a neighborhood of 0 in the reduced space $\mathbb{C}$ with a neighborhood of the image of $Q_j$, which we will denote by $Q_j$ as well, in $M_p$.

\begin{lemma}\label{Sympl(M_p)}
$\widetilde{\mathrm{Sympl}}(M_p, \{Q_j\}_{j=1}^3)$ is contractible. 
\end{lemma}
\begin{proof}
Step 1. \emph{A fibration $\widetilde{\mathrm{Sympl}}(M_p, \{Q_j\}_{j=1}^3)\rightarrow (Sp(4)^{S^1}/C_0)^3$.}\\
There is an obvious group homomorphism 
\begin{equation}\label{fib-S^1}
\widetilde{\mathrm{Sympl}}(M_p, \{Q_j\}_{j=1}^3)\rightarrow(Sp(4)^{S^1}/C_0)^3 ,
\end{equation}
by sending each automorphism to the tangent maps (modulo $C_0$)  at $0\in\mathbb{C}^2$ any of its lifting near $Q_i, i=1,2,3$ with respect to the fixed trivializations.  Let 
$$\widetilde{\text{Sympl}}_0(M_p, \{Q_j\}_{j=1}^3)=:\text{kernel of }(\ref{fib-S^1}),$$
 It is easy to see that (\ref{fib-S^1}) is a principal $\widetilde{\text{Sympl}}_0(M_p, \{Q_j\}_{j=1}^3)$-bundle. 
 
 Let 
 $$\text{Sympl}_0(M_p, \{Q_j\}_{j=1}^3):=\{\varphi\in\text{Sympl}(M_p, \{Q_j\}_{j=1}^3): (d\varphi)_{Q_j}=id, j=1,2,3\},$$
 where \index{$\text{Sympl}(M_p, \{Q_j\}_{j=1}^3)$}$\text{Sympl}(M_p, \{Q_j\}_{j=1}^3)$ is the true symplectomorphism group of $M_p$ fixing the three special points.
The next step shows that $\widetilde{\text{Sympl}}_0(M_p, \{Q_j\}_{j=1}^3)$ is homotopy equivalent 
$\text{Sympl}_0(M_p, \{Q_j\}_{j=1}^3)$.

Step 2.  $\widetilde{\text{Sympl}}_0(M_p, \{Q_j\}_{j=1}^3)\simeq \text{Sympl}_0(M_p, \{Q_j\}_{j=1}^3)$.\\
Let \index{$\mathrm{Sympl}^{T^{\check{u}_1}}_\sharp(M_{p,\nu}^{\underline{w}_1}, \{Q_j\}_{j=1}^3)$}$\mathrm{Sympl}^{T^{\check{u}_1}}_\sharp(M_{p,\nu}^{\underline{w}_1}, \{Q_j\}_{j=1}^3)$
be the subgroup in $\mathrm{Sympl}^{T^{\check{u}_1}}(M_{p,\nu}^{\underline{w}_1}, \{Q_j\}_{j=1}^3)$ consisting of elements $\hat{\phi}$ such that $\hat{\phi}$ restricted to a sufficiently small neighborhood of each $Q_j$ (within the fixed local chart) is the linear transformation $\begin{bmatrix}e^{-i\theta_j}& \\
 &e^{i\theta_j}
\end{bmatrix}$, for some $\theta_j\in [0,2\pi)$. Also let $\widetilde{\mathrm{Sympl}}_\sharp(M_p, \{Q_j\}_{j=1}^3)$ be the image of $\mathrm{Sympl}^{T^{\check{u}_1}}_\sharp(M_{p,\nu}^{\underline{w}_1}, \{Q_j\}_{j=1}^3)$ in $\widetilde{\mathrm{Sympl}}(M_p, \{Q_j\}_{j=1}^3)$.

Now we can construct a deformation retraction from the group $\widetilde{\text{Sympl}}_0(M_p, \{Q_j\}_{j=1}^3)$ to $\widetilde{\mathrm{Sympl}}_\sharp(M_p,\{Q_j\}_{j=1}^3)$. Near $Q_j$, the graph of  $\begin{bmatrix}e^{i\theta_j}& \\
 &e^{-i\theta_j}
\end{bmatrix}\circ\hat{\phi}$ is a Lagrangian in $(\mathbb{C}^2)^-\times\mathbb{C}^2\cong T^*\Delta_{\mathbb{C}^2}$, which is tangent to the zero section at $((Q_j, Q_j),0)\in T^*\Delta_{\mathbb{C}^2}$. Equivalently, in a smaller neighborhood of $(0,0)$, with respect to an appropriate Darboux coordinate system, it is the graph of the differential of a $\{\begin{bmatrix}e^{-i\theta}& \\
 &e^{i\theta}
\end{bmatrix}\}$-equivariant function $f_j$ with $Df_j(0)=0$ and $D^2f_j(0)=0$, where  $(Q_j, Q_j)$ is regarded as the origin in $\Delta_{\mathbb{C}^2}\cong \mathbb{C}^2$. Let $\mathrm{r}(\mathbf{z})=\|\mathbf{z}\|^2$ and fix a small ball $ B_j(\epsilon)=\{\mathrm{r}<\epsilon^2\}\subset \Delta_{\mathbb{C}^2}$, and let $\mathbb{D}_j(\frac{1}{16})\subset B_j(\epsilon)$ be the connected component containing 0 where $|D^2f_j|<\frac{1}{16}$. Here for a function $f$ on a domain, we adopt the following notations
$$|D^2f|=:\sup\limits_{x}\sum\limits_{m,n}|\frac{\partial^2f}{\partial x_m\partial x_n}(x)|,\ |Df|=:\sup\limits_{x}\sum\limits_n|\frac{\partial f}{\partial x_n}(x)|.$$ Now let $\epsilon_0=\sup\{\epsilon\in\mathbb{R}_+: B_j(\epsilon)\subset \mathbb{D}_j(\frac{1}{16})\} $, then we have $|Df_j|<\frac{1}{16}\epsilon_0$ and $|f_j|<\frac{1}{16}\epsilon_0^2$ on $B_j(\epsilon_0)$, if we make $f_j(0)=0$. Consider a $C^\infty$-function $b_{j,\epsilon_0}(x_1, x_2)$ on the square $[0, \epsilon_0^2)\times (-\frac{1}{16}\epsilon_0^2, \frac{1}{16}\epsilon_0^2)$ satisfying $b_{j,\epsilon_0}(x_1, x_2)=0$ for $|x_1|<\frac{1}{32}\epsilon_0^2$, $b_{j,\epsilon_0}(x_1,x_2)=x_2$ for $|x_1|>\frac{31}{32}\epsilon_0^2$, $b_{j,\epsilon_0}(x_1,0)=0$, and 
\begin{equation}\label{require}
|D^2b_{j,\epsilon_0}|(|D\mathrm{r}|^2+2|D\mathrm{r}|\cdot|Df_j|+|Df_j|^2)+|D_{x_1}b_{j,\epsilon_0}|\cdot|D^2\mathrm{r}|+|D_{x_2}b_{j,\epsilon_0}|\cdot |D^2f_j|<\frac{5}{6}.
\end{equation}
Then the graph of the differential of 
$$\begin{bmatrix}e^{-i\theta_j}& \\
 &e^{i\theta_j}
\end{bmatrix}\circ(s\cdot b_{j,\epsilon_0}\circ (\mathrm{r}, f_j)+(1-s)\cdot f_j)|_{B_j(\epsilon_0)}, 0\leq s\leq 1,$$
which is clearly $\{\begin{bmatrix}e^{-i\theta}& \\
 &e^{i\theta}
\end{bmatrix}\}$-equivariant, glues well with the graph of $\hat{\phi}$ outside of $B_j(\sqrt{\frac{31}{32}}\epsilon_0)$, and gives a family $\{\hat{\phi}_s\}_{s\in[0,1]}$ whose induced maps on $M_p$ lie in $\widetilde{\text{Sympl}}_0(M_p, \{Q_j\}_{j=1}^3)$, with $\hat{\phi}_0=\hat{\phi}$ and $\hat{\phi}_1\in \mathrm{Sympl}^{T^{\check{u}_1}}_\sharp(M_{p,\nu}^{\underline{w}_1},\{Q_j\}_{j=1}^3)$. Note that for (\ref{require}), if we start with $\epsilon$ small enough, then $|D_{x_1}b_{j,\epsilon_0}|<\frac{1}{12}$, $|D_{x_2}b_{j,\epsilon_0}|<2$ and $|D^2b_{j,\epsilon_0}|<5$, for instance, are sufficient for it to hold. We can fix such a small $\epsilon$ once for all, and make $b_{j,\epsilon_0}$ continuously depend on $\epsilon_0$ (in the $C^\infty$-topology). Thus we have a deformation retraction from $\widetilde{\text{Sympl}}_0(M_p, \{Q_j\}_{j=1}^3)$ to $\widetilde{\mathrm{Sympl}}_\sharp(M_p,\{Q_j\}_{j=1}^3)$. Similarly, we can easily show that $\mathrm{Sympl}_0(M_p, \{Q_j\}_{j=1}^3)$ deformation retracts onto $\widetilde{\mathrm{Sympl}}_\sharp(M_p,\{Q_j\}_{j=1}^3)$.

Step 3. \emph{$\widetilde{\mathrm{Sympl}}(M_p, \{Q_j\}_{j=1}^3)$ is contractible.} \\
 There is a natural fiber bundle
 \[\xymatrix{\text{Sympl}_0(M_p, \{Q_j\}_{j=1}^3)\ar[r] &\text{Sympl}(M_p, \{Q_j\}_{j=1}^3) \ar[d]\\
  &(Sp(2))^3
 }
 \]
 by the same construction as in (\ref{fib-S^1}).
 By standard results (c.f. \cite{EE67}), $\text{Sympl}(M_p, \{Q_j\}_{j=1}^3)$ is contractible, therefore $$B\text{Sympl}_0(M_p, \{Q_j\}_{j=1}^3)\simeq (Sp(2))^3\simeq (S^1)^3.$$
 In particular, the preimage of the fiber bundle over $(S^1)^3\cong (U(1))^3\subset (Sp(2))^3$, for which we will denote by \index{$\text{Sympl}^\dagger(M_p, \{Q_j\}_{j=1}^3)$}$\text{Sympl}^\dagger(M_p, \{Q_j\}_{j=1}^3)$, is homotopy equivalent to $\text{Sympl}(M_p, \{Q_j\}_{j=1}^3)$ via the inclusion. 
 
 On the other hand, there is an inclusion\footnote{To be more rigorous, one should replace $\text{Sympl}^\dagger(M_p, \{Q_j\}_{j=1}^3)$ by $\text{Sympl}^\dagger(M_p, \{Q_j\}_{j=1}^3)\cap \widetilde{\mathrm{Sympl}}(M_p, \{Q_j\}_{j=1}^3)$ for the inclusion, but the resulting space is homotopy equivalent to $\text{Sympl}^\dagger(M_p, \{Q_j\}_{j=1}^3)$, by the same technique in Step 2.} of the fibration involving $\text{Sympl}^\dagger(M_p, \{Q_j\}_{j=1}^3)$ into the fibration  (\ref{fib-S^1}). Since $Sp(4)^{S^1}/C_0\simeq S^1$ by Lemma \ref{angle} (c), using Step 2 and standard facts about classifying spaces, we deduce that $\widetilde{\mathrm{Sympl}}(M_p, \{Q_j\}_{j=1}^3)$ must be  contractible as well.

\end{proof}

Let 
$$\gamma_p: \mathrm{Sympl}^{T^{\check{u}_1}}(M_{p,\nu}^{\underline{w}_1}, \{Q_j\}_{j=1}^3)/C^\infty((-\nu,\nu), T^{\check{u}_1})\rightarrow\widetilde{\mathrm{Sympl}}(M_p, \{Q_j\}_{j=1}^3)$$
be the projection map. 
\begin{prop}\label{Sympl_contr}
$\gamma_p$ is a homotopy equivalence, hence $$\mathrm{Sympl}^{T^{\check{u}_1}}(M_{p,\nu}^{\underline{w}_1}, \{Q_j\}_{j=1}^3)/C^\infty((-\nu,\nu), T^{\check{u}_1})$$ is contractible. 
\end{prop}
\begin{proof}
First, we have the following commutative diagram 
\begin{equation}\label{fibrations}
\xymatrix{\mathrm{Sympl}^{T^{\check{u}_1}}(M_{p,\nu}^{\underline{w}_1}, \{Q_j\}_{j=1}^3)/C^\infty((-\nu,\nu), T^{\check{u}_1})\ar[d]^{ }  \ar[r]^{\ \ \ \ \ \ \ \ \ \ \ \ \ \gamma_p}&\widetilde{\mathrm{Sympl}}(M_p, \{Q_j\}_{j=1}^3)\ar[d]^{ }\\
(Sp(4)^{S^1}/C_0)^3\ar[r]^{id}&(Sp(4)^{S^1}/C_0)^3
},
\end{equation}
where the vertical arrows are both the restriction of the tangent maps at $Q_j, j=1,2,3$ (modulo $C$), and they give two fiber bundles. The kernel of the left map is the subgroup in  $\mathrm{Sympl}^{T^{\check{u}_1}}(M_{p,\nu}^{\underline{w}_1}, \{Q_j\}_{j=1}^3)/C^\infty((-\nu,\nu), T^{\check{u}_1})$ consisting of all liftings of elements in $\widetilde{\text{Sympl}}_0(M_p, \{Q_j\}_{j=1}^3)$ via $\gamma_p$.
The proof of Lemma \ref{Sympl(M_p)} shows that this group deformation retracts onto all liftings of elements in $\widetilde{\text{Sympl}}_\sharp(M_p, \{Q_j\}_{j=1}^3)$. 
We will apply the technique of real blow-ups to show that the latter subgroup deformation retracts onto $\widetilde{\text{Sympl}}_\sharp(M_p, \{Q_j\}_{j=1}^3)$. Therefore $\gamma_p$ is a homotopy equivalence. In the following, we will keep using the notations from the proof of Lemma \ref{Sympl(M_p)}.

For any $\phi\in \widetilde{\mathrm{Sympl}}_\sharp(M_p,\{Q_j\}_{j=1}^3)$, let $B_j, j=1,2,3$ be a small ball around $Q_j$ in $M_{p,\nu}^{\underline{w}_1}$ on which one of the liftings $\hat{\phi}$ is the linear transformation $\begin{bmatrix}e^{-i\theta_j}& \\
 &e^{i\theta_j}
\end{bmatrix}$ for some $\theta_j$. For $\epsilon, \delta>0$ small enough, the surgery for the real blow-up to $Bl_{\epsilon,\delta}(M_{p,\nu}^{\underline{w}_1})$ around each $Q_j$ is taken within a smaller ball $B'_j\subset B_j, j=1,2,3$ satisfying $\overline{B'}_j\subset B_j$, and we denote the resulting moment map for $T^{\check{u}_1}$ by 
\index{$\bar{\mu}_{\epsilon,\delta}$, $Bl_{\epsilon, \delta}(M_{p,\nu}^{\underline{w}_1})$}$$\bar{\mu}_{\epsilon,\delta}: Bl_{\epsilon, \delta}(M_{p,\nu}^{\underline{w}_1})\rightarrow \mathbb{R},$$
where $\bar{\mu}_{\epsilon,\delta}$ is regular over $(-\nu, \delta)$. Since we are only interested in $\bar{\mu}_{\epsilon,\delta}^{-1}(-\nu, \delta)$, in the following we will use the same notation  $Bl_{\epsilon, \delta}(M_{p,\nu}^{\underline{w}_1})$ to denote this submanifold. Clearly $\hat{\phi}$ induces a symplectomorphism $\hat{\phi}_{\epsilon,\delta}$ on $Bl_{\epsilon, \delta}(M_{p,\nu}^{\underline{w}_1})$, whose restriction to the blow-up region near $Q_j$ is the action by $\exp(\theta_j\check{u}_1)$. Conversely, given any $\hat{\phi}_{\epsilon,\delta}$ on $Bl_{\epsilon, \delta}(M_{p,\nu}^{\underline{w}_1})$ of this form, we can recover $\hat{\phi}$ on $(\mu_{p,\nu}^{\underline{w}_1})^{-1}(-\nu, \delta)$. 

Now we can describe the space of all liftings of $\widetilde{\mathrm{Sympl}}_\sharp(M_p,\{Q_j\}_{j=1}^3)$ in $\mathrm{Sympl}^{T^{\check{u}_1}}_\sharp(M_{p,\nu}^{\underline{w}_1}, \{Q_j\}_{j=1}^3)$ as a direct limit of spaces $X_{\epsilon,\delta}$ over $(\epsilon,\delta)\in(\mathbb{R_+})^2$, where we have a natural inclusion $X_{\epsilon_1,\delta_1}\hookrightarrow X_{\epsilon_2,\delta_2}$, when $\epsilon_1>\epsilon_2$ and $\delta_1>\delta_2$. The space $X_{\epsilon,\delta}$ consists of $\hat{\phi}$ whose restriction to a neighborhood of the three blow-up regions for $Bl_{\epsilon, \delta}(M_{p,\nu}^{\underline{w}_1})$ near each $Q_j$ is given by the action of $\exp(\theta_j\check{u}_1)$ for some $\theta_j\in\mathbb{R}$. By Proposition \ref{equiv_sympl}, after trivializing the reduced spaces of $\bar{\mu}_{\epsilon,\delta}$ over $(-\nu, \delta)$ and the reduced spaces\footnote{The trivialization for any $\delta$ determines a trivialization for all $\delta'<\delta$, so we can fix a uniform trivialization for $\delta$ less than a fixed $\delta_0$.} of $\mu_{p,\nu}^{\underline{w}_1}$ over $(0, \nu)$, we see that $X_{\epsilon,\delta}$ has a free $C^\infty((-\nu, \nu), T^{\check{u}_1})$-action, and $X_{\epsilon,\delta}/C^\infty((-\nu, \nu), T^{\check{u}_1})$ corresponds to the space of pairs of paths $(\rho_1, \rho_2)$, where $\rho_1: (-\nu, \delta)\rightarrow \mathrm{Sympl}(S^2)$, $\rho_2: (0, \nu)\rightarrow \mathrm{Sympl}(S^2)$ satisfy that $\rho_1$ restricts to the identity on a neighborhood of the blowing up loci and $\rho_1|_{(0,\delta)}$ is identified with $\rho_2|_{(0,\delta)}$ after the blowing down map. Then  $(\lim\limits_{\longrightarrow}X_{\epsilon,\delta})/C^\infty((-\nu, \nu), T^{\check{u}_1})$ deformation retracts onto $ \widetilde{\mathrm{Sympl}}_\sharp(M_p,\{Q_j\}_{j=1}^3)$, by deforming $\rho_1$ to the constant path determined by $\rho_1(0)$.

\end{proof}

\subsubsection{A deformation retraction for $\mathrm{ker}\beta_G$ supported near $T^*_\mathcal{B}\mathcal{B}$}

We start with a general set-up for the statement of Lemma \ref{normal} below. Let $(X,\omega_X)$ be a K$\ddot{\text{a}}$hler manifold, and $X^-\times X$ be equipped with the symplectic form $\omega_0=(-\omega_X)\times \omega_X$.  Let $N_{\Delta_X}$ be the normal bundle to the diagonal with respect to the K$\ddot{\text{a}}$hler metric $g\times g$ (which is the anti-diagonal in the tangent bundle restricted to $\Delta_X$). Then the product symplectic form gives a natural identification of $N_{\Delta_X}$ with $T^*\Delta_X$, thus induces a symplectic form $\omega_1$ on  $N_{\Delta_X}$. By the Lagrangian tubular neighborhood theorem, there is a symplectomorphism mapping a tubular neighborhood of the zero section in $T^*\Delta_X$ to a tubular neighborhood of $\Delta_X$ in $X^-\times X$, which fixes each point in $\Delta_X$. We state a slightly stronger statement in the following lemma.

\begin{lemma}\label{normal}
There exists a symplectomorphism $\psi$ from a tubular neighborhood $N^\epsilon_{\Delta_X}$ of the zero section in $N_{\Delta_X}$  to a tubular neighborhood $U_{\epsilon}(\Delta_X)$ of $\Delta_X$ in $X^-\times X$, such that $\psi|_{\Delta_X}=id$ and $d\psi|_{\Delta_X}=id$.
\end{lemma}

\begin{proof}
We first identify $N^\epsilon_{\Delta_X}$ with $U_{\epsilon}(\Delta_X)$ using the exponential map $\psi$ with respect to $g\times g$. Then it suffices to show that $\|\psi^*\omega_0-\omega_1\|_{((x,x), (tv,-tv))}\sim o(t)$ for any fixed $x$ and $v$, since by Moser's argument, the vector field generating an isotopy between $\psi^*\omega_0$ and $\omega_1$ will have length at most proportional to $o(t)$ in the direction of $v$, so the resulting diffeomorphism by integrating this vector field will have differential equal to the identity on the zero section. 

For any $v, u,w\in T_xX$, the push-forward of the vertical vector $(u,-u)$ and the horizontal lifting $(w(t), w(t))$ of $(w,w)\in T_{(x,x)}\Delta_X$ at $((x,x), (tv,-tv))$ to $X^-\times X$ under the exponential map is $((d\exp_x)|_{tv}(u), (d\exp_x)|_{(-tv)}(-u))$ and $(J_{w,v}(t), J_{w,-v}(t))$ respectively, where $J_{w,v}(t)$ denotes for the Jacobi vector field for the family of geodesics $\exp_{\exp_x(\tau w)}(t\Gamma(\exp_x(s w))_0^\tau(v))$, where $\Gamma(\exp_x(s w))_0^\tau$ means the parallel transport along the geodesic $\exp_x(s w)$ from time 0 to time $\tau$. 

The K$\ddot{\text{a}}$hler property implies that the covariant derivative $D_{\exp_x(tv)}(\omega)=0$, thus 
$$\omega(\Gamma(\exp_x(s v))_0^t(u), \Gamma(\exp_x(s v))_0^t(w))=\omega(u,w).$$ Now we only need to show that 
\begin{align*}
&\|(d\exp_x)|_{tv}(u)-\Gamma(\exp_x(s v))_0^t(u)\|\sim o(t),\\
& \|J_{w,v}(t)-\Gamma(\exp_x(s v))_0^t(w)\|\sim o(t),\\
&\omega_1(w_1(t), w_2(t))=\omega_1(w_1, w_2), \text{ for any two vectors }w_1,w_2\in T_x X.
\end{align*}
These properties hold for any Riemannian manifold $X$. In fact, one can take the geodesic coordinate at $x$, and use the fact that the Christoffel symbols vanish at $x$ to deduce that the covariant derivative of the first two of the above vectors along $\exp_x(tv)$ has norm $o(1)$. One can prove the last equality similarly.
\end{proof}

Let \index{$\mathrm{ker}\beta_G^\sharp$}$\mathrm{ker}\beta_G^\sharp$ be the subgroup of $\mathrm{ker}\beta_G$ consisting of $\varphi$ that restricts to the identity in a neighborhood of $T_\mathcal{B}^*\mathcal{B}$.
\begin{lemma}\label{id_zero}
\begin{itemize}
\item [(1)]For any $\varphi\in \ker\beta_G$, after sufficient conjugation by the conical dilations on $T^*\mathcal{B}$, the tangent space of $\mathrm{graph}(\varphi)$ can be made arbitrarily close to the tangent space of $\Delta_{T^*\mathcal{B}}$ along $\Delta_{\mathcal{B}}$.
\item [(2)] There is a deformation retraction from $\mathrm{ker}\beta_G$ to $\mathrm{ker}\beta_G^\sharp$. 
\end{itemize}
\end{lemma}
\begin{proof}
If we conjugate $\varphi$ by the dilation action $\delta_\lambda$ on $T^*\mathcal{B}$, i.e. we define
$$\varphi_\lambda(x, \xi)=\delta_{\lambda^{-1}}(\varphi(x, \lambda\xi))$$
then the limit as $\lambda\rightarrow 0^+$ of the tangent space of the graph of $\varphi_\lambda$ is the tangent space of $\Delta_{T^*\mathcal{B}}$ along $\Delta_\mathcal{B}$. To see this, we just need to check that for any curve $(x, t\xi), t\in [0,1]$, 
$$\lim\limits_{t\rightarrow 0}\frac{\mathrm{Dist}(\varphi_\lambda(x, t\xi),(x,t\xi))}{t}$$
uniformly approaches $0$ as $\lambda\rightarrow 0^+$, for all $\xi$ with $|\xi|=1$. Here $\mathrm{Dist}(-,-)$ denotes the distance between any two points with respect to any fixed metric on $T^*\mathcal{B}$. Let $(x(t), \xi(t))$ be a smooth family of representatives of $\varphi(x, t\xi)$ with $x(0)=x, \xi(0)=\xi$. Then we have $\varphi_\lambda(x, t\xi)=(x(\lambda t), \lambda^{-1}\xi(\lambda t))$. First, we have
$$\lim\limits_{t\rightarrow 0}\frac{\mathrm{Dist}(x(\lambda t),x)}{t}=\lim\limits_{t\rightarrow 0}\frac{\mathrm{Dist}(x(\lambda t),x)}{\lambda t}\lambda=|x'(0)|\lambda$$
(note that $x'(0)$ is regarded as an element in $\mathfrak{g}/\mathfrak{t}$), so this is uniformly approaching 0 as $\lambda\rightarrow 0^+$.
Second, using the fact that 
$$ta=:tx\xi x^{-1}=x(t)\xi(t) x(t)^{-1},$$
we have 
\begin{align*}
&\lim\limits_{t\rightarrow 0}\frac{\lambda^{-1}\xi(\lambda t)-t\xi}{t}=\lim\limits_{t\rightarrow 0}\frac{x(\lambda t)^{-1}ta x(\lambda t)-tx^{-1}ax}{t}\\
=&\lim\limits_{t\rightarrow 0}x(\lambda t)^{-1}a x(\lambda t)-x^{-1}ax=0.
\end{align*}
  Now (1) easily follows. 
(2) is a direct consequence of (1), Lemma \ref{normal} and a similar argument as in Lemma \ref{Sympl(M_p)} to give a deformation retraction. 


\end{proof}

\subsection{$\mathrm{ker}\beta_G$ is contractible}\label{main_pf}
This section is devoted to the proof that $\mathrm{ker}\beta_G$ is contractible. By Lemma \ref{id_zero}, we only need to prove that $\mathrm{ker}\beta_G^\sharp$ is contractible.
In the following, we identity $M_{s\cdot p}$ with $M_{p}$ for all $s>0$ using the $\mathbb{R}_+$-action, where $p$ is usually reserved for denoting any fixed element in $\mathbb{R}_+\cdot \underline{w}_1$, unless otherwise specified. 

Let $\sigma$ denote the projection of the subregular Springer fibers in $M_p$ (cf. Lemma \ref{codim1}), and let $T_0$ be the union of $\sigma$ with its image under the right Weyl group action on $T^*\mathcal{B}$ (induced from the right $\mathbf{W}$-action on $G/T$). Fix an open tubular neighborhood of $\sigma$ with a smooth boundary in $M_p$ and denote it by $\mathcal{U}_\sigma$. We assume that $T_0\pitchfork \partial\mathcal{U}_\sigma=\{P_1, P_3\}$, where $P_1$ (resp. $P_3$) is near $Q_1$ (resp. $Q_3$); see Figure \ref{shrink}. Using a fixed trivialization of the reduced spaces over $(-\nu, \delta_0)$ of the blow-up $Bl_{\epsilon, \delta_0}(M_{p,\nu}^{\underline{w}_1})$ as in the proof of Proposition \ref{Sympl_contr}, for some $\delta_0$ sufficiently small, we can choose a family of open sets $\mathcal{U}_{\sigma,c}$ with \index{$\mathcal{U}_{\sigma}$, $\mathcal{U}_{\sigma,c}$, $\phi_{X_{\sigma,c}}^{t}$}$\mathcal{U}_{\sigma,0}=\mathcal{U}_{\sigma}$ in $M_{p+c\cdot u_1}$ for $c\in (-\nu, \delta_0)$, after applying the blowing down map. We also choose a family of vector fields $X_{\sigma,c}$ on a neighborhood of $\overline{\mathcal{U}}_{\sigma, c}$ whose horizontal lifting to $M_{p,\nu}^{\underline{w}_1}$ can be lifted further to a smooth vector field on $Bl_{\epsilon, \delta_0}(M_{p,\nu}^{\underline{w}_1})$ which vanishes on $\{s\leq 0\}$ in the local coordinates as in Section \ref{blup}.1. We require that the time $t$ flow $\phi_{X_{\sigma,c}}^{t}$ of $X_{\sigma,c}$ scales the symplectic area of $\mathcal{U}_{\sigma,c}$ by $e^{-t}$, and when $c=0$, it is tangent to $T_0$ on the portion connecting $P_1$ (resp. $Q_3$) and $Q_1$ (resp. $P_3$) and deformation retracts $\mathcal{U}_{\sigma}$ onto $\sigma$ (see Figure \ref{shrink} below). 

Now let 
\begin{align*}
\mathcal{S}(\overset{\circ}{C_\nu}):=&\{\varphi\in \mathrm{Sympl}^T(\mu^{-1}(\overset{\circ}{C_\nu})): \varphi \text{ preserves the Springer fibers at infinity, }\\
&\varphi=id\text{ near the vertex of } C_\nu, \text{ and it is partially compactly supported}\}.
\end{align*} 
Similarly to $\beta_G$, we have a group homomorphism
$$\beta_{G, \overset{\circ}{C}_\nu}: \mathcal{S}(\overset{\circ}{C_\nu})\longrightarrow B_\mathbf{W}.$$

\begin{lemma}\label{horizontal}
The primitive $-\mathbf{p}d\mathbf{q}$ of $\omega$ vanishes on the subregular Springer fibers and their images under the right $\mathbf{W}$-action.
\end{lemma}
\begin{proof}
Consider the line segment $x=(\begin{bmatrix}\frac{1}{\sqrt{2}}&0&\frac{1}{\sqrt{2}}\\
\frac{\cos\theta}{\sqrt{2}}&-\sin\theta&-\frac{\cos\theta}{\sqrt{2}}\\
\frac{\sin\theta}{\sqrt{2}}&\cos\theta&-\frac{\sin\theta}{\sqrt{2}}
\end{bmatrix}, \xi=\begin{bmatrix}0&\cos\theta&\sin\theta\\ \cos\theta&0&0\\ \sin\theta&0&0\end{bmatrix}), \theta\in[0,\frac{\pi}{2}]$ in $\mu^{-1}(p_3)$. By a direct calculation, we see that $-\mathbf{p}d\mathbf{q}$ restricted to this segment is zero. It is also easy to check that $-\mathbf{p}d\mathbf{q}$ restricted to any $T^{\check{u}_1}$-orbit in $\mu^{-1}(p_3)$ vanishes. Therefore, using the 
invariance of $-\mathbf{p}d\mathbf{q}$ under the $G$-action and the right $\mathbf{W}$-action, we complete the proof. 
\end{proof}

\begin{remark}\label{s_sigma}
A direct consequence of Lemma \ref{horizontal} is one can take a smooth horizontal section of the $T^{\check{\underline{w}}_1}$-bundle over $(\mu_{p,\nu}^{\underline{w}_1})^{-1}(\bigcup\limits_{|c|<\nu}\mathcal{U}_{\sigma,c})$ containing a whole subregular Springer fiber and its image under the Weyl group action (intersecting the section), with respect to the connection form $-\mathbf{p}d\mathbf{q}$. We will fix such a section and denote it by \index{$\mathbf{s}_{\sigma}$}{$\mathbf{s}_{\sigma}$}.
\end{remark}

Suppose we are given a smooth path
$$\rho: (0,\infty)\rightarrow \mathrm{Sympl}^{T^{\check{u}_1}}(M_{p,\nu}^{\underline{w}_1}, \{Q_j\}_{j=1}^3)$$
such that for $s$ sufficiently large, $\rho(s)=id$ away from a neighborhood of the subregular Springer fibers. We assume that the neighborhood deformation retracts onto the subregular Springer fibers as $s\rightarrow \infty$, and we assume that when $s$ is sufficiently large, the induced automorphism on $M_{s\cdot p_3}$ from $\rho(s)$ restricted to $\mathcal{U}_\sigma$ corresponds to the identity element in $B_\mathbf{W}$. Then we have the following. 
\begin{lemma}\label{stretch}
Given any $\rho$ as above, we can stretch the parameter space $\mathbb{R}_+$ enough so that $\rho$ can be lifted to a symplectomorphism $\varphi_{\rho}$ in $\ker\beta_{G, \overset{\circ}{C}_\nu}$.
\end{lemma}
\begin{proof}
 Note that we can always have a lifting $\varphi_{\rho}$ such that $\varphi_{\rho}=id$ away from a neighborhood of the subregular Springer fibers when $s$ is sufficiently large, but we also need to make sure that in the limit $\varphi_\rho$ sends every subregular Springer fiber into itself but not to others in the $T^{\check{\underline{w}}_1}$-orbits. By Remark \ref{s_sigma} and the proof of Proposition \ref{equiv_sympl}, we can start from a lifting $\varphi_\rho$ (as an $T^{\check{\underline{w}}_1}$-equivariant diffeomorphism) such that that $(\varphi_\rho)_s$ preserves $\mathbf{s}_\sigma$, for $s$ sufficiently large, then by (\ref{equiv_3}) we modify $(\varphi_\rho)_s$ by the gauge transformation determined by a Hamiltonian function for the Hamiltonian vector field $\frac{d}{ds}\rho(s)$ (note that $H^1(M_{p,\nu}^{\underline{w}_1},\mathbb{R})=0$). Therefore, we only need that the integral of the length of the vector field $\frac{d}{ds}\rho(s)$ along $\rho(s)(\widetilde{T}_0)$ converges to zero as $s\rightarrow \infty$, where $\widetilde{T}_0$ is any smooth lifting of $T_0$ in $M_{p,\nu}^{\underline{w}_1}$. This can be achieved by sufficiently stretching the parameter space $\mathbb{R}_+$. 
\end{proof}

Now for every $\varphi\in \ker\beta_{G, \overset{\circ}{C}_\nu}$, it is determined by a smooth path 
$${\rho}_{\varphi,\nu}: (0,\infty)\rightarrow \mathrm{Sympl}^{T^{\check{u}_1}}(M_{p,\nu}^{\underline{w}_1}, \{Q_j\}_{j=1}^3)/C^\infty((-\nu,\nu), T^{\check{u}_1}),$$
up to the action of $\mathscr{C}(\overset{\circ}{C}_\nu, T)$. 
By Lemma \ref{stretch}, it is not hard to see that the space of such paths is homotopy equivalent to the space of paths satisfying the following properties:
\begin{enumerate}
 \item[(1)] $\rho_{\varphi,\nu}(s)=id$ for $s$ sufficiently small, 
 \item[(2)] whenever $s\geq N_0$, for some fixed integer $N_0>\frac{1}{\delta_0}$, the induced map of $\rho_{\varphi,\nu}(s)$ on the reduced spaces $M_{s\cdot (p+c\cdot u_1)}, c\in(-\nu, \nu)$, is the identity for $|c|\geq\frac{1}{s}$,
\item[(3)] for $s\geq N_0$ and $|c|<\frac{1}{s}$,  the induced map of $\rho_{\varphi,\nu}(s)$ on the reduced spaces $M_{s\cdot (p+c\cdot u_1)}$ is the identity outside $\phi_{X_{\sigma,c}}^{s}(\mathcal{U}_{\sigma,c})$, and it lies in the identity component of $\widetilde{\mathrm{Sympl}}^c(\mathbb{D},3\mathrm{pts})\simeq \mathrm{Sympl}^c(\mathbb{D}, 3\mathrm{pts})$, where we have chosen a symplectic identification between $(\mathcal{U}_\sigma, \{Q_j\}_{j=1}^3)$ and a 2-disc $\mathbb{D}$ with three marked points, and $\widetilde{\mathrm{Sympl}}^c(\mathbb{D},3\mathrm{pts})$ is defined similarly as $\widetilde{\mathrm{Sympl}}(M_p, \{Q_j\}_{j=1}^3)$ to indicate the special behavior of the automorphisms near the marked points. 
\end{enumerate}
A direct consequence of Proposition \ref{Sympl_contr} and the fact that $\mathrm{Sympl}^c(\mathbb{D})\simeq *$ is the following.

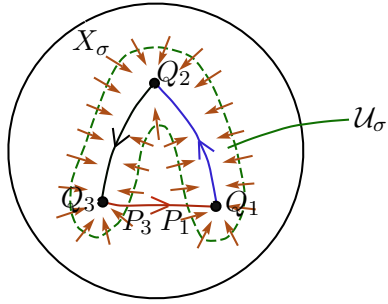
\begin{figure}
\centering
  \begin{overpic}[width=2in]{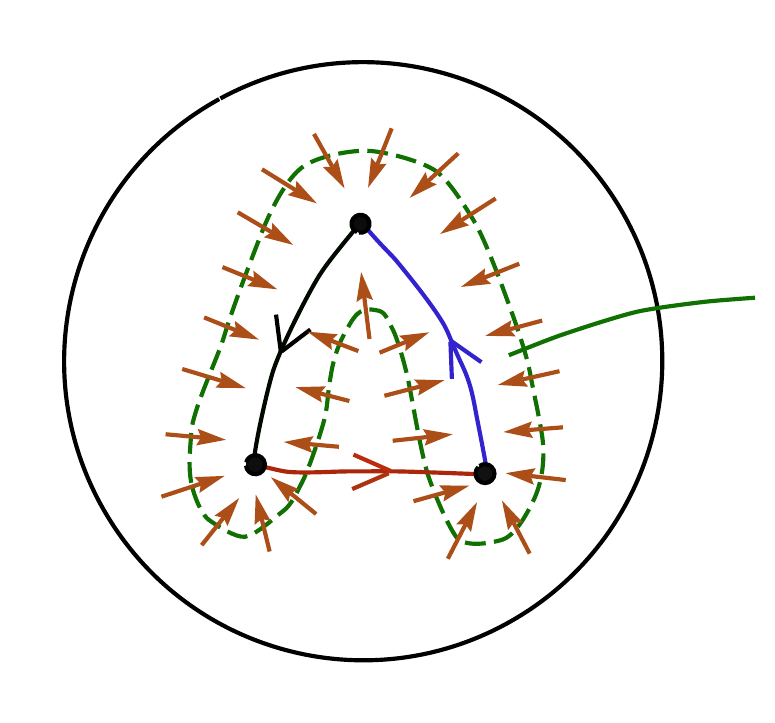}  
  \put(22,32){\small $Q_3$}
  \put(47,65){\small $Q_2$}
  \put(65,31){\small $Q_1$} 
  \put(99,53){\small $\mathcal{U}_\sigma$}
  \put(38,26){\small $P_3$}
  \put(48,26){\small $P_1$}
  \put(25,73){\small $X_\sigma$}  
  \end{overpic}
 \caption{A vector field $X_\sigma$ on $\mathcal{U}_\sigma$ shrinking it towards $\sigma$.} \label{shrink}
\end{figure}

\begin{lemma}\label{C_nu}
$\ker\beta_{G,\overset{\circ}{C}_\nu}/\mathscr{C}(\overset{\circ}{C}_\nu,T)$ is contractible.
\end{lemma}

\begin{lemma}\label{quot_triv}
The quotient map 
$$\ker\beta_{G,\overset{\circ}{C}_\nu}\rightarrow \ker\beta_{G,\overset{\circ}{C}_\nu}/\mathscr{C}(\overset{\circ}{C}_\nu,T)$$
is a trivial fiber bundle with fiber homotopy equivalent to the based loop space $\Omega_*(T)$.
\end{lemma}
\begin{proof}
The kernel of the map is obviously homotopy equivalent to $\Omega_*(T)$. We just need to show that there is a global section to the quotient map. 

For any $\varphi\in \ker\beta_{G,\overset{\circ}{C}_\nu}$, it induces an $\mathbb{R}_+$-family of $T^{\check{u}_1}$-invariant Hamiltonian vector fields on $M_{p,\nu}^{\underline{w}_1}$, where $p$ is any fixed element in $\mathbb{R}_+\cdot\underline{w}_1$, by differentiating its induced actions on $M_{s\cdot p,\nu}^{\underline{w}_1}$. So we get a family of Hamiltonian functions $H_{\varphi, s}$ on $M_{p,\nu}^{\underline{w}_1}$, by requiring that $H_{\varphi,s}(Z_0)=0$ for some chosen $Z_0\in (\mu_{p,\nu}^{\underline{w}_1})^{-1}(0)- \{Q_i\}_{i=1}^3$. On the other hand, given any $\overline{\varphi}\in \ker\beta_{G,\overset{\circ}{C}_\nu}/\mathscr{C}(\overset{\circ}{C}_\nu, T)$,
we can differentiate the actions on the reduced spaces of $T$ and get a family of Hamiltonian functions $\{H_{\overline{\varphi},s}\}_{s\in\mathbb{R}_+}$ up to the addition of a smooth function on $(-\nu,\nu)$ which has value $0$ at the origin. To fix this ambiguity, we can choose a smooth section $(-\nu, \nu)\rightarrow \bigcup\limits_{t\in (-\nu,\nu)}M_{p+t\cdot u_1}$ passing through $Z_0$, which avoids the singularities of the $T^{\check{u}_1}$-action and is submersive to the base, and require each $H_{\overline{\varphi},s}$ restricted to the zero function on this section. 

Once we have obtained $\{H_{\overline{\varphi},s}\}_{s\in\mathbb{R}_+}$, we can integrate their Hamiltonian vector fields along the radial directions and get a family of $T^{\check{u}_1}$-equivariant symplectomorphism $\widetilde{\varphi}_s$ on $M_{s\cdot p,\nu}^{\underline{w}_1}$. For $s\geq N_0$ (this is the bound as in condition (2)), on the complement to the preimage of $\phi_{X_{\sigma,c}}^s(\mathcal{U}_{\sigma,c})$ in $M_{p,\nu}^{\underline{w}_1}$, $\widetilde{\varphi}_s$ is given by an element in $C^\infty(\mathbb{R}_{(-\nu,\nu)}, T^{\check{u}_1})$. 
We can choose a global lifting of the restriction of $\widetilde{\varphi}_s$ there to $C^\infty(\mathbb{R}_{(-\nu,\nu)}, \mathbb{R}\cdot u_1)$, denoted by $\log\widetilde{\varphi}_s$, because $\ker\beta_{G,\overset{\circ}{C}_\nu}/\mathscr{C}(\overset{\circ}{C}_\nu,T)$ is contractible. Since we need $\widetilde{\varphi}_s=id$ (still on that complement) for $s$ sufficently large (we change the bound to $2N_0$), we can fix a smooth bump function $b: \mathbb{R}_+\rightarrow \mathbb{R}$, such that $b(s)=0$ for $s\leq N_0+1$ and $b(s)=1$ for $s\geq 2N_0-1$, and replace $\widetilde{\varphi}_s$ by its composition  with the exponential of $-b(s)\cdot \log\widetilde{\varphi}_s$. 
Alternatively, we can add to $\{H_{\overline{\varphi},s}\}_{s\in\mathbb{R}_+}$ a unique family of Hamiltonian functions $\{F_s(t)\}_{s\in\mathbb{R}_+}$ such that $F_s(0)=0$ for all $s$, and their Hamiltonian vector fields are $\frac{d}{ds}\exp(-b(s)\cdot \log\widetilde{\varphi}_s)$.

Next,  similarly to the proof of Lemma \ref{stretch}, we start from a $T$-equivariant diffeomorphism that lifts $\widetilde{\varphi}$ and preserves $\mathbf{s}_\sigma$, and then apply the procedure in the proof of Proposition \ref{equiv_sympl} to the family of Hamiltonian functions $\{H_{\overline{\varphi},s}+F_s(t)\}_{s\in\mathbb{R}_+}$, to get a $T$-equivariant symplectomorphism $\widetilde{\varphi}$ . We further repeat the process before to make $\varphi=id$ over $\{s\geq 2N_0\}$, away from the preimages of $\phi_{X_{\sigma,c}}^s(\mathcal{U}_{\sigma,c})$.
\end{proof}

\begin{prop}
$\mathrm{ker}\beta_G^\sharp$ is contractible. 
\end{prop}
\begin{proof}
In the following,  we use $p$ to denote for a fixed element in $\mathbb{R}_+\cdot w_0$, unless otherwise specified.

Step 1. \emph{A deformation retraction of $\ker\beta_G^\sharp$ supported on (the $G$-orbits of) $\mu^{-1}(\mathrm{Ad}_{G_{p}}(\overset{\circ}{W}_{\pm\epsilon}))$}

Note that $\mathbb{R}_+\cdot w_0$ can be viewed as the image of the moment map $\mu_{w_0}$ for the $T^{\frac{1}{3}iw_0}$-action on $\mu^{-1}(\mathrm{Ad}_{G_{p}}(\overset{\circ}{W}_{\pm\epsilon}))$. Since $\varphi|_{\mu^{-1}(p)}$ is the right multiplication by an element in $U(2)/\mu_3$,  $\varphi|_{\mu^{-1}(\mathbb{R}_+\cdot w_0)}$ corresponds to a loop in $U(2)/\mu_3$. On the other hand, given any $C^\infty$-map $\Lambda: \mathbb{R}_+\cdot w_0\rightarrow U(2)/\mu_3$ with $\Lambda(s)=1$ when $s$ is close to 0  and $\lim\limits_{s\rightarrow\infty}\Lambda(s)=1$, repeating the steps in Proposition \ref{equiv_sympl} from the right action by $\Lambda(p)$ on each reduced space along $\mathbb{R}_+\cdot w_0$, we get a canonical $U(2)$-equivariant symplectomorphism $\phi_\Lambda$ on $\mu^{-1}(\mathrm{Ad}_{G_p}(\overset{\circ}{W}_{\pm\epsilon}))$, such that $\phi_\Lambda|_{\mu^{-1}(p)}=R_{\Lambda(p)}, p\in \mathbb{R}_+\cdot w_0$, where $R_{\Lambda(p)}$ means the right multiplication by $\Lambda(p)$.

Now let $\Lambda_\varphi:  \mathbb{R}_+\cdot w_0\rightarrow U(2)/\mu_3$ be the loop defined by $\varphi$, then $\phi_{\Lambda_\varphi}^{-1}\varphi$ becomes the identity on $\mu^{-1}(\mathbb{R}_+\cdot w_0)$. By Lemma \ref{w_0}, and the fact that $T^*(U(2)/\mu_3)$ has a natural K$\ddot{\text{a}}$hler structure, we can run the same argument as in Lemma \ref{id_zero} to make $\phi_{\Lambda_\varphi}^{-1}\varphi$ isotopic to the identity in a neighborhood of $\mu^{-1}(\mathbb{R}_+\cdot w_0)$. Then composing it back with $\phi_{\Lambda_\varphi}$, we get a symplectomorphism of $\mu^{-1}(\mathrm{Ad}_{G_p}(\overset{\circ}{W}_{\pm\epsilon}))$, which agrees with $\varphi$ near the boundary and is $\phi_{\Lambda_\varphi}$ near $\mu^{-1}(\mathbb{R}_+\cdot w_0)$. Since $\varphi$ is the identity map near the infinity, the region of $\phi_{\Lambda_\varphi}$ contains a conical neighborhood of $\mu^{-1}(\mathbb{R}_+\cdot w_0)$, and we can push it to contain a fixed conical neighborhood of $\mu^{-1}(\mathbb{R}_+\cdot w_0)$, say $\mu^{-1}(\overset{\circ}{W}_{\pm \frac{\epsilon}{2}})$. Similarly, we can deform $\varphi$ over $\mu^{-1}(\mathrm{Ad}_{G_{w_2}}(\overset{\circ}{W'}_{\pm \frac{\epsilon}{2}}))$ in the same way, where $W'_{\pm \frac{\epsilon}{2}}$ is the cone bounded by $\mathbb{R}_{\geq 0}\cdot (w_2\pm \frac{\epsilon}{2}\cdot \mathrm{diag}(0,1,-1))$. 

From now on, we can restrict ourselves to the space of $\varphi\in \ker \beta_G^\sharp$ where $\varphi$ restricted to $\mu^{-1}(\mathrm{Ad}_{G_{w_0}}(\overset{\circ}{W}_{\pm \frac{\epsilon}{2}}))$ is $\phi_{\Lambda_\varphi}$ and it has similar behavior over $\mu^{-1}(\mathrm{Ad}_{G_{w_2}}(\overset{\circ}{W'}_{\pm \frac{\epsilon}{2}}))$. For simplicity, we still denote this space by $\ker\beta_G^\sharp$.

Step 2. \emph{ $\ker\beta_G^\sharp$ is contractible.}
Any $\varphi\in \ker\beta_G^\sharp$ is determined by its restriction to $\mu^{-1}(\mathrm{Ad}_{G_{w_0}}(\overset{\circ}{W}_{\pm \frac{\epsilon}{2}}))$, $\mu^{-1}(\mathrm{Ad}_{G_{w_2}}(\overset{\circ}{W'}_{\pm \frac{\epsilon}{2}}))$, $\mu^{-1}(\overset{\circ}{W}_{01})$, $\mu^{-1}(\overset{\circ}{W}_{12})$ and $\mu^{-1}(\overset{\circ}{C}_\nu)$, with the obvious matching conditions. Using Lemma \ref{quot_triv}, we have $\ker\beta_G^\sharp$ is homotopy equivalent to the fiber product
\begin{align}\label{fiberprod}
&\Omega_*(U(2)/\mu_3)\times_{\mathrm{Sympl}^T(\mu^{-1}(\overset{\circ}{W}_{01}))}\Omega_*(T)\times_{\mathrm{Sympl}^T(\mu^{-1}(\overset{\circ}{W}_{12}))} \Omega_*(U(2)/\mu_3).
\end{align}
Here $\mathrm{Sympl}^T(\mu^{-1}(\overset{\circ}{W}_{ij}))$ (with the obvious restriction  as before on the vertex and infinity of $W_{ij}$) is homotopy equivalent to a fiber bundle over 
$\Omega_*(\mathrm{Sympl}(S^2))\simeq \Omega_*(SO(3))$ (c.f. \cite{Smale}) for $(i,j)=(0,1)$ and $(1,2)$ with the fiber homotopy equivalent to $\Omega_*(T)$. The restriction map $\Omega_*(T)\rightarrow \mathrm{Sympl}^T(\mu^{-1}(\overset{\circ}{W}_{ij}))$ in (\ref{fiberprod}) is homotopic to the inclusion as the fiber over the constant loop. On the other hand the restriction maps $\Omega_*(U(2)/\mu_3)\rightarrow \mathrm{Sympl}^T(\mu^{-1}(\overset{\circ}{W}_{ij}))$ for $(i,j)=(0,1)$ and $(1,2)$ are respectively induced from (in the homotopic sense) the inclusion of the sequence $\Omega_*(T^{\frac{1}{3}iw_0})\rightarrow\Omega_*(U(2)/\mu_3)\rightarrow \Omega_*(SO(3))$ (this sequence is a fibration if we replace $\Omega_*(SO(3))$ by the image of $\Omega_*(U(2)/\mu_3)$ under the quotient map), coming from the quotient map $U(2)/\mu_3\rightarrow SO(3)$ by the center $T^{\frac{1}{3}iw_0}$, via the inclusion $T^{\frac{1}{3}iw_0}\hookrightarrow T$, and $\Omega_*(T^{\frac{1}{3}iw_2})\rightarrow\Omega_*(U(2)/\mu_3)\rightarrow \Omega_*(SO(3))$ via the inclusion $T^{\frac{1}{3}iw_2}\hookrightarrow T$.
Therefore, the resulting space is contractible. 
\end{proof}

\printindex

\end{document}